\documentclass[12pt]{article}
\usepackage{amsfonts}

\usepackage[latin1]{inputenc}
\usepackage{amsmath,amssymb}
\usepackage{latexsym}
\usepackage[active]{srcltx}
\usepackage[
bookmarks=true,         %generates bookmarks for all entries in the "Table of Contents"
bookmarksnumbered=true, %includes chapter-/header-/section-/subsection-/... -
colorlinks=true, pdfstartview=FitV, linkcolor=blue, citecolor=blue,
urlcolor=blue]{hyperref}

\newtheorem{theorem}{Theorem}[section]
\newtheorem{corollary}[theorem]{Corollary}
\newtheorem{lemma}[theorem]{Lemma}
\newtheorem{proposition}[theorem]{Proposition}

\newtheorem{remark}[theorem]{Remark}
\numberwithin{equation}{section}
\def\R{{\mathbb R}}
\def\E{{{\mathbb E}\,}}

\def\N{{\mathbb N}}

\def\Var{{\mathop {{\rm Var\, }}}}

\def\square{{\vcenter{\vbox{\hrule height.3pt
        \hbox{\vrule width.3pt height5pt \kern5pt
           \vrule width.3pt}
        \hrule height.3pt}}}}

\def\tlint{{- \kern-0.85em \int \kern-0.2em}}
\def\dlint{{- \kern-1.05em \int \kern-0.4em}}

\def\si{\sigma}

\def\si{{\sigma}}

\def\si{{\sigma}}

\def \eref#1{\hbox{(\ref{#1})}}

\def \eref#1{\hbox{(\ref{#1})}}

\def\si{{\sigma}}

\newenvironment{proof}[1][Proof]{\noindent\textit{#1.} }{\hfill \rule{0.5em}{0.5em}}

\begin{document}

\title{Limit theorems for additive functionals of some self-similar Gaussian processes}
\date{\today}
\author{Minhao Hong, Heguang Liu and Fangjun Xu\thanks{F. Xu is partially supported by National Natural Science Foundation of China (Grant No.11871219).}
}

\maketitle

\begin{abstract}

Under certain mild conditions, limit theorems for additive functionals of some $d$-dimensional self-similar Gaussian processes are obtained. These limit theorems work for general Gaussian processes including fractional Brownian motions, sub-fractional Brownian motions and bi-fractional Brownian motions. To prove these results, we use the method of moments and an enhanced chaining argument. The Gaussian processes under consideration are required to satisfy certain strong local nondeterminism property. A tractable sufficient condition for the strong local nondeterminism property is given and it only relays on the covariance functions of the Gaussian processes. Moreover, we give a sufficient condition for the distribution function of a random vector to be determined by its moments.


\end{abstract}

\section{Introduction}
Let  $X=\{X_t=(X^1_t,\dots, X^d_t),\, t\geq 0\}$ be a $d$-dimensional centered Gaussian process whose components are independent copies of some one-dimensional centered Gaussian process. We assume that $X$ satisfies the self-similarity with index $H\in (0,1)$ and the following strong local nondeterminism property: there exists a positive constant $\kappa_H$ such that for any integer $m\ge 1$, any times $0=s_0<s_1\leq \dots \leq
s_m<t<\infty$,  
\begin{equation}
\Var\Big(X_t| X_{s_1},\dots, X_{s_m}\Big)\geq \kappa_H \min_{1\leq j\leq m} |t-s_j|^{2H}. \label{slndp}
\end{equation} 
According to Remark  1.4 in \cite{hx}, if $Hd<1$,  using Lemma 3.8 in \cite{clrs} and similar arguments as in the proof of Theorem 1.1 in \cite{hx}, we could easily obtain that the local time $L_t(\lambda)$ of $X$ exists in $L^p$ for any $p\in [1,\infty)$. Moreover, for any $\alpha\in (0, \frac{1-Hd}{H}\wedge 1)$, $\mathbb{E} |L_t(z+x)-L_t(x)|^2$ is less than a constant multiple of $|z|^{2\alpha}$ for all $z,x\in\mathbb{R}^d$. Therefore, for any integrable function $f:\R^d \rightarrow \R$ satisfying $\int_{\R^d}|f(x)| |x|^{\beta } dx<\infty$ with $\beta>0$, one can show that
\begin{align*}
&\mathbb{E}\left|n^{Hd} \int_0^{t} f(n^H(X_s-\lambda))\, ds-L_t(\lambda) \int_{\R^d} f(x)\, dx\right|\\
&=\mathbb{E} \left| \int_{\mathbb{R}^d} f(x) (L_t(\frac{x}{n^H}+\lambda)-L_t(\lambda)) dx\right|\\
&\leq \int_{\mathbb{R}^d} |f(x)| \left(\mathbb{E} |L_t(\frac{x}{n^H}+\lambda)-L_t(\lambda)|^2\right)^{1/2} dx\\
&\leq c_{\alpha}\, n^{-H\alpha}  \int_{\mathbb{R}^d} |f(x)||x|^{\alpha}  dx
\end{align*}
for any $\alpha\in (0, \beta\wedge \frac{1-Hd}{H}\wedge 1)$. This implies 
\begin{align*}
&\lim_{n\to\infty}\mathbb{E}\left|n^{Hd} \int_0^{t} f(n^H(X_s-\lambda))\, ds-L_t(\lambda) \int_{\R^d} f(x)\, dx\right|=0.
\end{align*}
In this paper, we will investigate the asymptotic behavior of 
\begin{align} \label{main}
n^{Hd} \int_0^{t} f(n^H(X_s-\lambda))\, ds-L_t(\lambda) \int_{\R^d} f(x)\, dx
\end{align}
as $n$ tends to infinity.  When $X$ is the Brownian motion or the fractional Brownian motion, the asymptotic behavior of (\ref{main}) received a lot of attention in the past decades, see the recent paper \cite{jnnp} and the references therein. In \cite{jnnp}, the authors considered the asymptotic behavior of (\ref{main}) when $X$ is the one-dimensional fractional Brownian motion. Convergence of finite dimensional distributions of (\ref{main}) in the case $H\geq \frac{1}{3}$ and the $L^2$ convergence of (\ref{main}) in the case $H<\frac{1}{3}$ are obtained there.

We will extend the limit theorems in \cite{jnnp} to more general multi-dimensional Gaussian processes with the convergence of finite dimensional distributions improved to the convergence in law in the space $C([0,\infty ))$. Moreover, the influence of the function $f$ on the range of the index $H$ will be clearly characterized. For example, if $f$ is the density of the standard real normal random variable, then the constant $\mathcal{C}_{\frac{1}{3}}(f)$ in \cite{jnnp} is equal to zero and so the limiting process in Theorem 1.3 in \cite{jnnp} is the zero process when $H=\frac{1}{3}$. This means that the limit theorems in \cite{jnnp} could be further optimized according to the function $f$. In fact,  with some proper modifications, Theorem 1.3 in \cite{jnnp} is still true for $H\geq \frac{1}{5}$, see our Theorems \ref{th1} and \ref{th2} below. Therefore, our results have better applications in the statistical inference for these Gaussian processes.

In order to formulate our results, we need to introduce some notation.  Fix a number $\beta\in (0, 2]$,  let 
\[
\Theta^{\beta}=\Big\{f \in L^{2}(\R^d): \int_{\R}|f(x)| |x|^{\beta } dx<\infty\Big\}
\] 
for $\beta\in (0,1]$ and 
\[
\Theta^{\beta}=\Big\{f \in L^{2}(\R^d): \int_{\R^d}|f(x)| |x|^{\beta } dx<\infty, \int_{\R^d} xf(x)dx=0\Big\} 
\]
for $\beta\in(1,2]$. It is easy to see that $f\in\Theta^{\beta}$ implies $f\in L^{1}(\R^d)$.

We also need the following assumptions on the Gaussian process $X$:
\begin{itemize}
\item[({\bf A})]  For any $t>0$, there exist constants $ \eta\geq 1$,  $\sigma>0$ and non-negative decreasing function $\phi(\varepsilon)$  on $[0,1/\eta]$ with $\lim\limits_{\varepsilon\to 0}\phi(\varepsilon)=0$, such that
\[
0\leq h^{2H}(\sigma-\phi(h/t))\leq \mathrm{Var}(X^1_{t+h}-X^1_t)\leq h^{2H}(\sigma+\phi(h/t))
\]
for all $h\in[0,t/\eta]$.
 
\item[({\bf B})] For any $0<t_1<t_2<t_3<t_4<\infty$ and $\eta>1$, there exists a non-negative decreasing function $\psi(\eta)$ with $\lim\limits_{\eta\to\infty}\psi(\eta)=0$ such that,
if  (i) $\frac{\Delta t_2}{\Delta t_4}\leq \frac{1}{\eta}$ or (ii) $\frac{\Delta t_2}{\Delta t_4} \geq \eta$ or (iii) $\frac{\max\{\Delta t_2, \Delta t_4\}}{\Delta t_3}\leq \frac{1}{\eta}$, then 
 \[
\Big|\E\big(X^1_{t_4}-X^1_{t_3}\big)\big(X^1_{t_2}-X^1_{t_1}\big)\Big|\leq  \psi(\eta)\, \left[\E\big(X^1_{t_4}-X^1_{t_3}\big)^2\right]^{\frac{1}{2}}\E\left[\big(X^1_{t_2}-X^1_{t_1}\big)^2 \right]^{\frac{1}{2}},
\] 
where $\Delta t_i=t_i-t_{i-1}$ for $i=2,3,4$.
\end{itemize}

By Lemmas 2.4-2.6 in \cite{sxy}, Assumptions $({\bf A}) \&({\bf B})$ are satisfied by fractional Brownian motions, bi-fractional Brownian motions and sub-fractional Brownian motions. Moreover, we only use the covariance functions of these Gaussian processes in the proofs of these lemmas. 

A sufficient condition for a centered self-similar Gaussian process to satisfy the strong local nondeterminism (\ref{slndp}) is given in Theorem \ref{glndp} in the appendix.  As a consequence, we obtain that bi-fractional Brownian motions and sub-fractional Brownian motions satisfy the strong local nondeterminism property (\ref{slndp}), see Corollaries \ref{bfbm} and \ref{sfbm}.  Since fractional Brownian motions are special bi-fractional Brownian motions, they also satisfy the strong local nondeterminism (\ref{slndp}).  A very simple  proof for fractional Brownian motions to satisfy  (\ref{slndp}) is available in \cite{clrs}. The strong local nondeterminism in Theorem \ref{glndp} is different from those in \cite{tx,l} where the times are in a finite closed interval bounded away from $0$ and the constant in the lower bound should depend on the interval. We don't have such requirements on the times and the constant $\kappa_{H}$ in the lower bound. Moreover, the hypothesis in Theorem \ref{glndp} only relays on the covariance function of the Gaussian process.

For any $H\in (0,1)$, let $G^{H,d}_L$ be the class of all $d$-dimensional centered Gaussian processes satisfying the self-similarity with index $H$, the strong local nondeterminism property (\ref{slndp}) and Assumptions $({\bf A}) \&({\bf B})$. Clearly, the class $G^{H,d}_L$ contains fractional Brownian motions, bi-fractional Brownian motions and sub-fractional Brownian motions. In particular, to verify that a self-similar Gaussian process belongs to the class $G^{H,d}_L$, we only need to utilize its covariance function. This will make the verification not difficult. In general,  using the method of moments to prove central limit theorems for additive functionals of Gaussian processes is very complicated. A chaining argument was first developed in \cite{nx} to simplify the kind of proofs.  However, it could not clearly reflect the influence of the function $f$ on the limit theorems. So we develop an enhanced chaining argument to overcome this drawback. See Propositions \ref{chain1} and \ref{chain2} below for the enhanced chaining argument.

The following are main results of this paper.
\begin{theorem}
\label{th1} Suppose $X\in G^{H,d}_L$, $\frac{1}{2\beta+d}<H<\frac{1}{d}$ and $f\in \Theta^{\beta}$ for some $\beta\in(0,2]$.
Then, for any $\lambda\in\mathbb{R}^d$,
\begin{align*}
 & \left\{ n^{\frac{Hd+1}{2}} \Big(\int^t_0 f(n^H(X_s-\lambda))ds-n^{-Hd}L_t(\lambda)\int_{\mathbb{R}^d} f(x)dx\Big),\; 
  t\ge 0\right\} \\
  &\qquad\qquad\qquad\qquad\qquad\qquad\qquad\qquad\qquad \overset{\mathcal{L}}{\longrightarrow }  \left\{ \sqrt{C_{H,d,f}} W(L_t(\lambda)) \,, t\ge
  0\right\}
\end{align*}
in the space $C([0,\infty ))$, as $n$ tends to infinity,  where ``$\overset{\mathcal{L}}{\longrightarrow }$'' denotes the convergence in law, $W $ is a real-valued standard Brownian
motion independent of $X$,
\begin{align*}
C_{H,d, f} 
&=\frac{2}{(2\pi)^d}   \int_0^\infty \int_{\mathbb{R}^d}\Big|\widehat{f}(y)-\widehat{f}(0)\Big|^2e^{-\frac{\sigma}{2}|y|^2 s^{2H}}dyds%=\left(\frac{2}{\sigma}\right)^{\frac{1}{2H}}\frac{\Gamma(\frac{1}{2H})}{H(2\pi)^d} \int_{\mathbb{R}^d}|\widehat{f}(y)-\widehat{f}(0)|^2|y|^{-\frac{1}{H}}dy,
\end{align*}
with $\widehat{f}$ denoting the Fourier transform of $f$ and $\sigma$ being the positive constant in Assumption $({\bf A})$. % and $\Gamma(\cdot)$ is the Gamma function.
\end{theorem}

\begin{theorem}
\label{th2} Suppose $X\in G^{H,d}_L$, $H=\frac{1}{2\beta+d}$ and $f\in \Theta^{\beta}$ for some $\beta\in \{1, 2\}$. %and 
%\begin{align} \label{cbetad}
%c_{\beta, d}:=\beta^{1-\beta} \lim_{|y|\to 0} |y|^{-\beta}\Big|\int_{\mathbb{R}^d} f(u)\sin^{\beta}(u\cdot y) du\Big|
%\end{align}
%exists. 
Then, for any $\lambda\in\mathbb{R}^d$,
\begin{align*}
 & \left\{(\ln n)^{-\frac{1}{2}} n^{\frac{Hd+1}{2}} \Big(\int^t_0 f(n^H(X_s-\lambda))ds-n^{-Hd}L_t(\lambda)\int_{\mathbb{R}^d} f(x)dx\Big),\; 
  t\ge 0\right\} \\
  &\qquad\qquad\qquad\qquad\qquad\qquad\qquad\qquad\qquad\qquad\qquad \overset{\mathcal{L}}{\longrightarrow } \ \
  \Big\{  \sqrt{D_{H,d,f}} W(L_t(\lambda)) \,, t\ge
  0\Big\}
\end{align*}
in the space $C([0,\infty ))$, as $n$ tends to infinity,  where  $W $ is a real-valued standard Brownian
motion independent of $X$ and
\begin{align}\label{DHdf}
D_{H,d, f}   \nonumber
&=\frac{2}{(2\pi)^d} \lim_{n\to\infty}\frac{1}{\ln n}   \int_0^{n} \int_{\mathbb{R}^d}\Big|\widehat{f}(y)-\widehat{f}(0)\Big|^2e^{-\frac{\sigma}{2}|y|^2s^{2H}}dyds \\ 
&=\frac{2}{(2\pi)^d \sigma^{\frac{1}{2H}}} \int_{\mathbb{R}^d}\Big| \int_{\mathbb{R}^d} f(z) \frac{1}{\beta}(z\cdot x)^{\beta} dz\Big|^2e^{-\frac{1}{2}|x|^2}dx.
\end{align}
\end{theorem}

\begin{theorem}
\label{th3} Suppose $X\in G^{H,d}_L$, $0<H<\frac{1}{2\beta+d}$, $f\in\Theta^{\beta}$ for some $\beta\in \{1,2\}$. Then, for any $t>0$, $\lambda\in\mathbb{R}^d$ and $p\geq 1$,
\begin{align*}
 &n^{H\beta} \Big(n^H\int^t_0 f(n^H(X_s-\lambda))ds-L_t(\lambda)\int_{\mathbb{R}^d} f(x)dx\Big) \\
&\qquad\qquad \overset{L^p}{\longrightarrow}
\begin{cases}
\displaystyle\int_{\mathbb{R}^d} \Big(\sum\limits^d_{i=1} L^{({\bf e}_i)}_t(\lambda)x_i\Big) f(x)\, dx  &  \text{if}\; \beta=1;  \\   
\displaystyle\int_{\mathbb{R}^d} \Big( \frac{1}{2}\sum\limits^d_{i,k=1} x_i L^{({\bf e}_i+{\bf e}_k)}_t(\lambda)x_k\Big) f(x)\, dx  & \text{if}\; \beta=2,
 \end{cases}
\end{align*}
where ${\bf e}_i$ is the $i$-th unit vector in $\mathbb{R}^d$, $L^{({\bf e}_i)}_t(\lambda)$ is the partial derivative of the local time $L_t(\lambda)$ along the direction ${\bf e}_i$ and $L^{({\bf e}_i+{\bf e}_k)}_t(\lambda)$ is the partial derivative of the local time $L_t(\lambda)$ along the direction ${\bf e}_i+{\bf e}_k$, see \cite{hx} for definitions of these derivatives of $L_t(\lambda)$.
\end{theorem}

\begin{remark} The limit in (\ref{DHdf}) exists when $f\in\Theta^{\beta}$ for some $\beta\in \{1,2\}$ and $1-Hd=2H\beta$. In fact, making proper change of variables and applying the L'H$\hat{o}$pital's rule,
\begin{align*}
&\lim_{n\to\infty}\frac{1}{\ln n}   \int_0^{n} \int_{\mathbb{R}^d}\Big|\widehat{f}(y)-\widehat{f}(0)\Big|^2e^{-\frac{\sigma}{2}|y|^2s^{2H}}dyds \\ 
&=\frac{1}{\sigma^{\frac{1}{2H}}}\lim_{t\to\infty}\frac{1}{\ln t}   \int_0^{t} \int_{\mathbb{R}^d}\Big|\widehat{f}(y)-\widehat{f}(0)\Big|^2e^{-\frac{1}{2}|y|^2 u^{2H}}dydu\\
&=\frac{1}{\sigma^{\frac{1}{2H}}}\lim_{t\to\infty}   t\int_{\mathbb{R}^d}\Big|\widehat{f}(y)-\widehat{f}(0)\Big|^2e^{-\frac{1}{2}|y|^2t^{2H}}dy\\
&=\frac{1}{\sigma^{\frac{1}{2H}}}\lim_{t\to\infty}   t^{1-Hd}\int_{\mathbb{R}^d}\Big|\widehat{f}(\frac{x}{t^H})-\widehat{f}(0)\Big|^2e^{-\frac{1}{2}|x|^2}dx\\
&=\frac{1}{\sigma^{\frac{1}{2H}}}\lim_{t\to\infty}  \int_{\mathbb{R}^d}\Big| \int_{\mathbb{R}^d} f(z)t^{H\beta}(e^{\iota z\cdot \frac{x}{t^H}}-1) dz\Big|^2e^{-\frac{1}{2}|x|^2}dx,
\end{align*}
where $\iota=\sqrt{-1}$.  

For the case $\beta=1$,  by the dominated convergence theorem, 
\begin{align*}
\lim_{t\to\infty}  \int_{\mathbb{R}^d}\Big| \int_{\mathbb{R}^d} f(z)t^{H}(e^{\iota z\cdot \frac{x}{t^H}}-1) dz\Big|^2e^{-\frac{1}{2}|x|^2}dx=\int_{\mathbb{R}^d}\Big| \int_{\mathbb{R}^d} f(z) (z\cdot x) dz\Big|^2e^{-\frac{1}{2}|x|^2}dx.
\end{align*}
For the case $\beta=2$,  by $\int_{\mathbb{R}^d} uf(u) du=0$ and the dominated convergence theorem,  
\begin{align*}
&\lim_{t\to\infty}  \int_{\mathbb{R}^d}\Big| \int_{\mathbb{R}^d} f(z)t^{2H}(e^{\iota z\cdot \frac{x}{t^H}}-1) dz\Big|^2e^{-\frac{1}{2}|x|^2}dx\\
&=\lim_{t\to\infty}  \int_{\mathbb{R}^d}\Big| \int_{\mathbb{R}^d} f(z)t^{2H}(e^{\iota z\cdot \frac{x}{t^H}}-1-\iota z\cdot \frac{x}{t^H}) dz\Big|^2e^{-\frac{1}{2}|x|^2}dx\\
&= \int_{\mathbb{R}^d}\Big| \int_{\mathbb{R}^d} f(z) \frac{(z\cdot x)^2}{2} dz\Big|^2e^{-\frac{1}{2}|x|^2}dx.
\end{align*}
Combining these results gives the equality (\ref{DHdf}).

\end{remark}

\begin{remark}
With some proper modifications, we believe that Theorems \ref{th2} and \ref{th3} are also true for other $\beta$s in $(0,2]$. However,  this is more complicated than the case $\beta\in \{1,2\}$. For example, if $f$ is the density function of a $d$-dimensional rotationally invariant $\beta$-stable random variable with $\beta\in (0,2)$, then $f$ does not satisfy $\int_{\mathbb{R}^d} |f(x)||x|^{\beta}dx<\infty$. To obtain similar results as in Theorems \ref{th2} and \ref{th3}, we need to replace $\int_{\mathbb{R}^d} |f(x)||x|^{\beta}dx<\infty$ in the definition of $\Theta^\beta$ by $\int_{\mathbb{R}^d} |f(x)||x|^{\beta-\epsilon}dx<\infty$ for any $\epsilon\in (0,\beta)$. If $f$ is the density function of a symmetric random variable in the domain of attraction of the $\beta$-stable law with $\beta\in (0,2)$, we should not only make the modification in the definition of $\Theta^\beta$ but also revise the normalizing factors. See \cite{IL} for properties of random variables in the domain of attraction of stable law.  
\end{remark}

After some preliminaries in Section 2,  in Section 3  we develop an enhanced chaining argument. Section 4 is devoted to the 
proofs of Theorems  \ref{th1}, \ref{th2} and \ref{th3}. Finally, a tractable sufficient condition for the strong local nondeterminism property is given in the Appendix A and a sufficient condition for the distribution function of a random vector to be determined by its moments in the Appendix B. Throughout this paper, if not mentioned otherwise, the letter $c$, 
with or without a subscript, denotes a generic positive finite constant whose exact value is independent of $n$ and may change from line to line. Moreover, we use $x\cdot y$ to denote the usual inner product in $\mathbb{R}^d$ and $\iota$ to denote $\sqrt{-1}$.

\section{Preliminaries}

In this section, we will give a few lemmas which are important in the proofs of our main results. Recall that $X=\{X_t=(X^1_t,\dots, X^d_t),\, t\geq 0\}$ is a $d$-dimensional centered Gaussian process whose components are independent copies of some one-dimensional centered Gaussian process. The first lemma shows that the self-similarity, strong local nondeterminism property (\ref{slndp}) and Assumption $({\bf A})$ imply a local nondeterminism property. 
\begin{lemma} \label{lema1} 
Suppose that $X$ satisfies the self-similarity with index $H\in (0,1)$, the strong local nondeterminism property (\ref{slndp}) and Assumption $({\bf A})$, then there exists a positive constant $\kappa_{H,m}$ such that for any integer $m\ge 1$, any times $0=s_0<s_1\leq \dots \leq s_m<\infty$ and $x_i\in\mathbb{R}^d$ with $1\leq i\leq m$, 
\begin{equation} \label{lndp}
\Var\Big(\sum^m_{i=1} x_i\cdot (X_{s_i}-X_{s_{i-1}})\Big)\geq \kappa_{H,m}\sum^m_{i=1} |x_i|^2(s_i-s_{i-1})^{2H}.
\end{equation}
\end{lemma}
\begin{proof}
By the self-similarity and Assumption $({\bf A})$, there exists a positive constant $c_1$ such that $\mathbb{E}|X_t-X_s|^2\leq c_1(t-s)^{2H}$ for any $t,s\geq 0$. Then, using similar arguments as in the proof of Theorem 2.6 in \cite{n} gives the desired result.
\end{proof}

The next lemma gives a formula for moments of the increments of the process  
 $\left\{W (L_{t}(\lambda)), t\ge 0\right\}$ on disjoint intervals, where $W$ is a real-valued standard Brownian motion
independent of the Gaussian process $X$.
\begin{lemma} \label{lema2} Fix a finite number of disjoint intervals 
$(a_i, b_i]$ in $[0,\infty)$, where $i = 1, \dots, N$  and $b_i \leq a_{i+1}$. 
Consider a multi-index $\mathbf{m} = (m_1, \dots, m_N)$, where $m_i \geq 1$ and 
$1 \leq  i \leq N$. Then 
\[
\E \Big(\prod_{i=1}^N \big[W (L_{b_i} (\lambda) ) - W (L_{a_i} (\lambda))\big]^{m_i} \Big)
\]
is equal to 
\[
\bigg(\prod\limits_{i=1}^N \frac{m_i!}{ 2^{\frac{m_i}{2}}   (2\pi )^{\frac{m_id}{2} }   (\frac{m_i}{2})!} \bigg)\displaystyle \int_{\prod\limits_{i=1}^N [a_i ,b_i]^{\frac{m_i}{2}} }\displaystyle \int_{\mathbb{R}^{\frac{|\mathbf{m}|d}{2}}}  \exp\Big(-\iota \lambda\cdot \sum\limits^N_{i=1} \sum\limits^{\frac{m_i}{2}}_{j=1}x^j_i-\frac{1}{2}\Var\big(\sum\limits^N_{i=1}\sum\limits^{\frac{m_i}{2}}_{j=1}x^j_i\cdot X_{u^j_i}\big)\Big)\, dx\, du 
\]
if all $m_i$ are even and $0$ otherwise, where $|\mathbf{m}|=\sum\limits^N_{i=1} m_i$. Moreover, if all $m_i$ are even, then 
\[
\E \Big(\prod_{i=1}^N \big[W (L_{b_i} (\lambda) ) - W (L_{a_i} (\lambda))\big]^{m_i} \Big)\leq \Big(\frac{b_N^{1-H}\Gamma(1-Hd)}{2(2\pi \kappa_{H})^{\frac{d}{2} }} \Big)^{\frac{|\mathbf{m}|}{2}}  \Big(\prod\limits_{i=1}^N  \frac{m_i!}{\Gamma\big(\frac{m_i}{2}(1-Hd)+1\big)} \Big),
\]
where $\Gamma(\cdot)$ is the Gamma function.
\end{lemma}

\begin{proof}
It is easy to see that when one of $m_i$ is odd, then the expectation is $0$. Suppose now that all $m_i$  are even. Denote by $\mathcal{F}^X$  the $\si$-algebra generated by the Gaussian process $X$.  Since $W$ is a standard Brownian motion independent of $X$,
\begin{align*}
&\E \prod_{i=1}^N  \Big[W (L_{b_i} (\lambda) ) - W (L_{a_i} (\lambda)) \Big]^{m_i} \\
&=\E \bigg[\E\Big( \prod_{i=1}^N [W (L_{b_i} (\lambda) ) - W (L_{a_i} (\lambda)) ]^{m_i} \big|\mathcal{F}^X\Big) \bigg]\\
&=\Big( \prod_{i=1}^N \frac{m_i!}{ 2^{\frac{m_i}{2}} (\frac{m_i}{2})!} \Big)
\E\prod_{i=1}^N  [ L_{b_i} (\lambda)  -  L_{a_i} (\lambda) ]^{\frac{m_i}{2}}\\
&=\Big(\prod\limits_{i=1}^N \frac{m_i!}{ 2^{\frac{m_i}{2}}   (2\pi )^{\frac{m_id}{2} }   (\frac{m_i}{2})!} \Big)\int_{\prod\limits_{i=1}^N [a_i ,b_i]^{\frac{m_i}{2}} } \int_{\mathbb{R}^{\frac{|\mathbf{m}|d}{2}}}  \exp\Big(-\iota \lambda\cdot \sum\limits^N_{i=1} \sum^{\frac{m_i}{2}}_{j=1}x^j_i\Big) \\
&\qquad\qquad\qquad\times \exp\Big(-\frac{1}{2}\Var\big(\sum\limits^N_{i=1}\sum^{\frac{m_i}{2}}_{j=1}x^j_i\cdot X_{u^j_i}\big)\Big)\, dx\, du.
\end{align*}

When all $m_i$ are even,
\begin{align*}
&\E \prod_{i=1}^N  \Big[W (L_{b_i} (\lambda) ) - W (L_{a_i} (\lambda))\Big]^{m_i} \\
&\leq \bigg(\prod\limits_{i=1}^N \frac{m_i!}{ 2^{\frac{m_i}{2}}   (2\pi )^{\frac{m_id}{2} }   (\frac{m_i}{2})!} \bigg)\int_{\prod\limits_{i=1}^N [a_i ,b_i]^{\frac{m_i}{2}} } \int_{\mathbb{R}^{\frac{|\mathbf{m}|d}{2}}}  \exp\Big(-\frac{1}{2}\Var\big(\sum\limits^N_{i=1}\sum^{\frac{m_i}{2}}_{j=1}x^j_i\cdot X_{u^j_i}\big)\Big)\, dx\, du\\
%&=\Big(\prod\limits_{i=1}^N \frac{m_i!}{ 2^{\frac{m_i}{2}}   (2\pi )^{\frac{m_id}{4} }   (\frac{m_i}{2})!} \Big)
%\displaystyle \int_{\prod\limits_{i=1}^N [a_i ,b_i]^{\frac{m_i}{2}}} \det(A(w))^{-1/2} dw\\
%&\leq \Big(\prod\limits_{i=1}^N \frac{m_i!}{ 2^{\frac{m_i}{2}}   (2\pi )^{\frac{m_id}{4} }   (m_i/2)!} \Big)  \Big(\prod\limits_{i=1}^N  (m_i/2)!  \frac{\Gamma^{m_i}(1-H)}{\Gamma(m_i(1-H)-1)} \Big) (\kappa_{H,m})^{-\frac{1}   {2}\sum\limits^N_{i=1} m_i}  (b_N-a_1)^{(1-H)\sum\limits^N_{i=1} m_i}
&\leq \bigg(\frac{b_N^{1-Hd}\Gamma(1-Hd)}{ 2 (2\pi \kappa_H)^{\frac{d}{2} }} \bigg)^{\frac{|\mathbf{m}|}{2}}  \Big(\prod\limits_{i=1}^N  \frac{m_i!}{\Gamma\big(\frac{m_i}{2}(1-Hd)+1\big)} \Big),
\end{align*}
where in the last inequality we use the strong local nondeterminism property (\ref{slndp}), Lemma 3.8 in \cite{clrs} and Lemma A.2. in \cite{hx}. 
\end{proof}

As a consequence, with the help of the Stirling's formula $\lim\limits_{a\to\infty}\frac{\Gamma(a+1)}{a^a e^{-a} \sqrt{2\pi a}}=1$ and Theorem \ref{moment} in the Appendix B,  the law of the random vector  $\big(W (L_{b_i} (\lambda) ) - W (L_{a_i} (\lambda)): 1\le i\le N\big)$ is determined by
the moments computed in the above lemma.
 
 The last two lemmas are needed in obtaining the enhanced chaining argument, see Propositions \ref{chain1} and \ref{chain2}. Lemma \ref{f} is used in the proof of Proposition \ref{chain1} and both are used in the proof of Proposition \ref{chain2}.

\begin{lemma} \label{f}
Suppose that $f\in \Theta^{\beta}$. If $\beta\in (0,1]$, then, for any $\alpha\in [0,\beta]$, there exists a positive constant $c_{\alpha,1}$ depends on $\alpha$ such that 
\[
|\widehat{f}(x_1-x_2)-\widehat{f}(-x_2)|\leq c_{\alpha,1}(|x_1|^{\alpha}\wedge 1)
\] 
for all $x_1,x_2\in\mathbb{R}^d$. If $\beta\in (1,2]$, then for any $\alpha\in [\theta,\beta]$ and $\theta\in [0,1]$, there exists a positive constant $c_{\alpha,2}$ depends on $\alpha$ such that 
\[
|\widehat{f}(x_1-x_2)-\widehat{f}(-x_2)|\leq c_{\alpha,2}\Big((|x_1|^{\alpha}+|x_1|^{\theta}|x_2|^{\alpha-\theta})\wedge 1\Big)
\] 
for all $x_1,x_2\in\mathbb{R}^d$. 
\end{lemma}

\begin{proof} The case $\beta\in (0,1]$ is well-known. We only consider the case $\beta\in (1,2]$. In this case, $\int_{\R^d} xf(x)dx=0$. For $\alpha\in [\theta,1]$,  it is easy to see that $|\widehat{f}(x_1-x_2)-\widehat{f}(-x_2)|$ is less than a constant multiple of $|x_1|^{\alpha}\wedge 1$.  For $\alpha\in (1,\beta]$, 
\begin{align*}
&\left|\widehat{f}(x_1-x_2)-\widehat{f}(-x_2)\right|\\
&=\Big| \int_{\R^d} f(u) (e^{\iota u\cdot (x_1-x_2)}-e^{-\iota u\cdot x_2}) du\Big|\\
&=\Big| \int_{\R^d} f(u) (e^{\iota u\cdot (x_1-x_2)}-e^{-\iota u\cdot x_2}-\iota u\cdot x_1) du\Big|\\
&\leq \Big| \int_{\R^d} f(u) e^{-\iota u\cdot x_2}(e^{\iota u \cdot x_1}-1-\iota u\cdot x_1)du\Big|+\Big| \int_{\R^d}  \iota u\cdot x_1 f(u)(e^{-\iota u\cdot x_2}-1) du\Big|\\
&\leq c_1\big(|x_1|^{\alpha}+|x_1||x_2|^{\alpha-1}\big).
\end{align*}
If $|x_2|\leq |x_1|$, then $|\widehat{f}(x_1-x_2)-\widehat{f}(-x_2)|\leq 2c_1|x_1|^{\alpha}$. If $|x_2|>|x_1|$, then $|x_1||x_2|^{\alpha-1}\leq |x_1|^{\theta}|x_2|^{\alpha-\theta}$ for any $\theta\in[0,1]$. Note that $|\widehat{f}(x_1-x_2)-\widehat{f}(-x_2)|$ is uniformly bounded. Combining all these estimates gives the desired result.
\end{proof}

\begin{lemma} \label{int}
For any $T>0$, $n\in\mathbb{N}$, $1-Hd=2H\alpha$, $\alpha>0$ and $p=0,1,2$, there exists a positive constant $c_{T,d,H,\alpha}$ depending on $T$, $d$, $H$ and $\alpha$ such that
\begin{align*}
\int^{nT}_0\int_{\mathbb{R}^d} (|x|^{\alpha}\wedge 1)^p e^{-|x|^2u^{2H}} dx\,du
&\leq c_{T,d,H,\alpha} \begin{cases}
(nT)^{1-Hd-Hp \alpha}  & \text{if}\;1-Hd>Hp\alpha; \\   
\ln n & \text{if}\; 1-Hd=Hp\alpha.
\end{cases}
\end{align*}
\end{lemma}

\begin{proof}
Note that 
\begin{align*}
\int^{nT}_0\int_{\mathbb{R}^d} (|x|^{\alpha}\wedge 1)^p e^{-|x|^2u^{2H}} dx\,du
&=\int^{nT}_0\int_{|x|\leq 1} |x|^{p\alpha} e^{-|x|^2u^{2H}} dx\,du+\int^{nT}_0\int_{|x|>1} e^{-|x|^2u^{2H}} dx\,du\\
&\leq c_1\Big(T+\int^{nT}_T \int_{|x|\leq 1} |x|^{p\alpha} e^{-|x|^2u^{2H}} dx\,du+1\Big)\\
&\leq c_1\Big(T+1+\int^{nT}_T u^{-Hd-Hp\alpha}\,du\Big).
\end{align*}
The desired result follows from some simple calculations.
\end{proof}

\section{An enhanced chaining argument}

In this section, we will give the enhanced chaining argument for moments of the increment of the stochastic process $F_n=\{F_n(t):\; t\geq 0\}$ where
\[
F_n(t)=\ell_{n,H,d}\, \Big(\int^t_0 f(n^H(X_s-\lambda))ds-n^{-Hd}L_t(\lambda)\int_{\mathbb{R}} f(x)dx\Big)
\]
with
\begin{align} \label{switch}
\ell_{n,H, d}=
 \begin{cases}
 n^{\frac{Hd+1}{2}} & \text{if} \quad H>\frac{1}{2\beta+d} \\ \\
 (\ln n)^{-\frac{1}{2}} n^{\frac{Hd+1}{2}}   & \text{if} \quad H=\frac{1}{2\beta+d}.
 \end{cases}
\end{align}

For any $0\leq a<b<\infty$ and $m\in\N$, let  
\[
I^n_m=\E\left[ (F_n(b) -F_n(a))^m\right].
\] 
Clearly,
\begin{align*}
I_m^n&= c_{m,d}\, \ell^m_{n,H,d} \,  n^{-m}  \int_{ \R ^{md}}  \int_{D_m} \prod_{i=1}^m(\widehat{f}(y_i)-\widehat{f}(0))\\
&\qquad\qquad\qquad\qquad \times \exp\Big(-\iota n^H \lambda \cdot \sum^m_{i=1}y_i-\frac{1}{2}\Var\big(\sum\limits^m_{i=1} y_i\cdot X_{s_i}\big)\Big)\, ds\, dy,
 \end{align*}
where $c_{m,d}=\frac{m!}{(2\pi)^{md}}$ and $D_m=\big\{(s_1, \dots, s_m): na<s_1<\cdots<s_m<nb\big\}$. 

Making the change of variables
$x_i =\sum\limits_{j=i}^m y_j$ (with the convention that $x_{m+1} =0$), we can write
\begin{align*}
I^n_m
&=c_{m,d}\, \ell^m_{n,H,d} \, n^{-m}\int_{\R^{md}}\int_{D_m} \Big(\prod^m_{i=1} (\widehat{f}(x_i-x_{i+1})-\widehat{f}(0))\Big)\\
&\qquad\qquad\qquad \times \exp\Big(-\iota n^H\lambda \cdot x_1-\frac{1}{2}\Var\big(\sum\limits^m_{i=1} x_i\cdot (X_{s_i}-X_{s_{i-1}})\big)\Big)\, ds\, dx.
\end{align*}
The main idea in order to estimate these terms is to replace each product  
\[
(\widehat{f}(x_{2i-1}-x_{2i})-\widehat{f}(0))(\widehat{f}(x_{2i}-x_{2i+1})-\widehat{f}(0))
\]    
by $(\widehat{f}( -x_{2i})-\widehat{f}(0))(\widehat{f}(x_{2i})-\widehat{f}(0))= |\widehat{f}(  x_{2i})-\widehat{f}(0)|^2 $. Then, the differences  $\widehat{f}(x_{2i-1}-x_{2i})- \widehat{f}( -x_{2i})$ and
$\widehat{f}(x_{2i}-x_{2i+1})- \widehat{f}(  x_{2i})$ are bounded by the corresponding estimates in Lemma \ref{f}.  We are going to make these substitutions recursively. To do this, we introduce the following notation.

Let $I^n_{m,0}=I^n_m$. For $k=1,\dots,m$, we define
\begin{align*}
I^n_{m,k}
&=c_{m,d}\,  \ell^m_{n,H,d} \,  n^{-m}   \, \int_{\R^{md}}\int_{D_m} I_k\, \prod^m_{i=k+1} (\widehat{f}(x_i-x_{i+1})-\widehat{f}(0))\\
&\qquad\qquad\times \exp\Big(-\iota n^H\lambda \cdot x_1-\frac{1}{2}\Var\big(\sum\limits^m_{i=1} x_i\cdot (X_{s_i}-X_{s_{i-1}})\big)\Big)\, ds\, dx,
\end{align*}
where 
\[
I_k =
\begin{cases}
\prod\limits^{\frac{k-1}{2}}_{j=1}|\widehat{f}(x_{2j})-\widehat{f}(0)|^2 (\widehat{f}(-x_{k+1})-\widehat{f}(0)), & \text{if } k \text{ is odd}; \\ 
\prod\limits^{\frac{k}{2}}_{j=1}|\widehat{f}(x_{2j})-\widehat{f}(0)|^2, & \text{if }k\text{ is even}.
\end{cases}
\]
The following Propositions \ref{chain1} and \ref{chain2} control the difference between $I^n_{m,k-1}$ and $ I^n_{m,k}$. One is for the case $\frac{1}{2\beta+d}<H<\frac{1}{d}$ with $\beta\in (0,2]$. The other is for the case $H=\frac{1}{2\beta+d}$ with $\beta\in \{1,2\}$.  These two propositions give the enhanced chaining argument for the terms $I^n_{m,k}$ with $k=0,1,\cdots,m$. The original chaining argument for this kind of terms was developed in \cite{nx} and corresponds to the case $\beta=1$ in Proposition \ref{chain1} below. 

For $H\in (\frac{1}{2\beta+d}, \frac{1}{d})$, let
\begin{equation} \label{gamma0}
\gamma_0=\min\Big(\frac{1-Hd }{2}, \frac{2H\beta-1+Hd}{2} \Big).
\end{equation}
Obviously, $\gamma_0=\frac{1-Hd }{2}$ if $1-Hd\leq H\beta$ and  $\gamma_0=\frac{2H\beta-1+Hd}{2}$ if $1-Hd>H\beta$. 

Fix a positive constant $\gamma$ such that $\gamma\in (0,\gamma_0)$.

\begin{proposition} \label{chain1} Suppose that $\frac{1}{2\beta+d}<H<\frac{1}{d}$ and $f\in \Theta^{\beta}$ for some $\beta\in(0,2]$. For $k=1,2,\dots,m$, there exists a positive constant $c$, which depends on $\gamma$, such that
\[
|I^n_{m,k-1}-I^n_{m,k}|\leq c\, n^{-\gamma} (b-a)^{\frac{m(1-Hd)}{2}-\gamma}.
\]
\end{proposition}
 
\begin{proof}  We divide the proof into two parts. 

\noindent
{\bf Part I}: We consider the case $\beta\in (0,1]$. Suppose that  $k$ is odd. Applying the local nondeterminism property  (\ref{lndp}) and making the change of variable $u_1=s_1$, $u_i =s_i -s_{i-1}$, for $2\le i\le m$, we can show that $|I^n_{m,k-1}-I^n_{m,k}|$ is less than a constant multiple of
\begin{align*}
&n^{m\frac{Hd-1}{2}}\int_{\R^{md}}\int_{O_m} \Big(\prod^m_{i=k+1}  |\widehat{f}(x_i-x_{i+1})-\widehat{f}(0)|\Big)\\
&\qquad\times \big|\widehat{f}(x_k-x_{k+1})-\widehat{f}(-x_{k+1})\big|\Big(\prod^{\frac{k-1}{2}}_{j=1}|\widehat{f}(x_{2j})-\widehat{f}(0)|^2\Big)\, \exp\Big(-\frac{\kappa_{H,m}}{2} \sum\limits^m_{i=1} |x_i|^2u^{2H}_i\Big)\, du\, dx,
\end{align*}
where $O_{m}=\big\{ (u_1, \dots, u_m):   0<u_i \le n(b-a),    i=1,\dots,m\big\}$.

Integrating with respect to $x_i$s and $u_i$s with $i\leq k-1$ and then using Lemma \ref{f} give
\begin{align*}
&|I^n_{m,k-1}-I^n_{m,k}|\\
&\leq c_1\, (b-a)^{\frac{k-1}{2}(1-Hd) }\, n^{(m-k+1)\frac{Hd-1}{2}}\int_{\R^{(m-k+1)d}}\int_{O_{m,k}} |x_{k}|^{\alpha_k}|x_m|^{\alpha_m} \prod^{m-1}_{i=k+1}  (|x_i|^{\alpha_i}+|x_{i+1}|^{\alpha_i})\\
&\qquad\qquad\times \exp\Big(-\frac{\kappa_{H,m}}{2} \sum\limits^m_{i=k} |x_i|^2u^{2H}_i\Big)\, d\overline{u}_k\, d\overline{x}_k\\
&\leq c_1\, (b-a)^{\frac{k-1}{2}(1-Hd) }\, n^{(m-k+1)\frac{Hd-1}{2}}\sum_{\mathcal{S}_k}\int_{\R^{(m-k+1)d}}\int_{O_{m,k}} |x_{k}|^{\alpha_k} |x_m|^{\alpha_m} \prod^{m-1}_{i=k+1}  (|x_i|^{p_i\alpha_i}|x_{i+1}|^{\overline{p}_i\alpha_i})\\
&\qquad\qquad\times \exp\Big(-\frac{\kappa_{H,m}}{2} \sum\limits^m_{i=k} |x_i|^2u^{2H}_i\Big)\, d\overline{u}_k\, d\overline{x}_k\\
&\leq c_2\, (b-a)^{\frac{k-1}{2}(1-Hd) }\, n^{(m-k+1)\frac{Hd-1}{2}}  \sum_{\mathcal{S}_k} \int_{O_{m,k}} u_k^{-Hd-H\alpha_k}u_{k+1}^{-Hd-Hp_{k+1}\alpha_{k+1}} u_m^{-Hd-H(\overline{p}_{m-1}\alpha_{m-1}+\alpha_m)} \\
&\qquad\qquad\qquad\times \prod^{m-1}_{i=k+2} u_i^{-Hd-H(\overline{p}_{i-1}\alpha_{i-1}+p_i\alpha_i)}\, d\overline{u}_k,
\end{align*}
where $\alpha_i$s are constants in $[0,\beta]$, $\mathcal{S}_k=\big\{p_i, \overline{p}_i:  p_i\in \{0,1\}, p_i+\overline{p}_i=1,\, i=k+1,\dots,m-1\big\}$, $d\overline{u}_k=du_k\cdots du_{m}$, $d\overline{x}_k=dx_k\cdots dx_{m}$ and $O_{m,k}=\{ (u_k, \dots, u_m): 0<u_i \le n(b-a),    i=k,\dots,m\}$.

For any $\varepsilon\in (0,\frac{1-Hd}{2H})$, let
\[
\alpha_k=\begin{cases}
\frac{1-Hd}{H}-\varepsilon   & \text{if } 1-Hd\leq H\beta; \\ \\
\beta-\varepsilon & \text{if } H\beta<1-Hd<2H\beta,
\end{cases}
\]
and $\alpha_i=\frac{1-Hd}{2H}-\varepsilon$ for $i=k+1,\cdots, m$. 

With these choices of $\alpha_i$s, integrating with respect to $u_i$s with $i\geq k$ gives 
\begin{align*}
&|I^n_{m,k-1}-I^n_{m,k}|\leq c_3\, (b-a)^{\frac{k-1}{2}(1-Hd) }\, n^{(m-k+1)\frac{Hd-1}{2}} (nb-na)^{(m-k+1)(1-Hd)-H\sum\limits^m_{i=k}\alpha_i}.
\end{align*}
Observe that
\begin{align*}
(m-k+1)\frac{Hd-1}{2}+(m-k+1)(1-Hd)-H\sum^m_{i=k}\alpha_i=-\gamma_0+ (m-k+1)H\varepsilon.
\end{align*}
Choosing $\varepsilon$ small enough gives 
\[
|I^n_{m,k-1}-I^n_{m,k}|\leq c_4\, n^{-\gamma} (b-a)^{\frac{m(1-Hd)}{2}-\gamma}.
\]

Suppose now that $k$ is even. Using similar notation and arguments as above, we can show that
\begin{align*}
&|I^n_{m,k-1}-I^n_{m,k}|\\
&\leq c_5 \, n^{m\frac{Hd-1}{2}}\int_{\R^{md}}\int_{O_m} \Big(\prod^{\frac{k-2}{2}}_{j=1}|\widehat{f}(x_{2j})-\widehat{f}(0)|^2\Big) \big|\widehat{f}(-x_k)-\widehat{f}(0)\big|\big|\widehat{f}(x_k-x_{k+1})-\widehat{f}(x_k)\big| \\
&\qquad\qquad\qquad\times \prod^m_{i=k+1}  |\widehat{f}(x_i-x_{i+1})-\widehat{f}(0)|\, \exp\Big(-\frac{\kappa_{H,m}}{2} \sum\limits^m_{i=1} |x_i|^2u^{2H}_i\Big)\, du\, dx\\
&\leq c_6\, (b-a)^{\frac{k}{2}(1-Hd) }\, n^{(m-k)\frac{Hd-1}{2}} \int_{\R^{(m-k+1)d}}\int_{O_{m,k}}  |x_k|^{\alpha_{k-1}}|x_{k+1}|^{\alpha_k} |x_m|^{\alpha_m}\\
&\qquad\qquad\qquad\times \Big(\prod^{m-1}_{i=k+1}  (|x_i|^{\alpha_i}+|x_{i+1}|^{\alpha_i})\Big) \exp\Big(-\frac{\kappa_{H,m}}{2} \sum\limits^m_{i=k} |x_i|^2u^{2H}_i\Big)\, d\overline{u}_k\, d\overline{x}_k\\
&= c_6\, (b-a)^{\frac{k}{2}(1-Hd) }\, n^{(m-k)\frac{Hd-1}{2}}  \sum_{\mathcal{S}_k} \int_{\R^{(m-k+1)d}}\int_{O_{m,k}}  |x_k|^{\alpha_{k-1}}|x_{k+1}|^{\alpha_k} |x_m|^{\alpha_m}\\
&\qquad\qquad\qquad \times  \Big(\prod^{m-1}_{i=k+1}  |x_i|^{p_i\alpha_i}|x_{i+1}|^{\overline{p}_i\alpha_i}\Big) \exp\Big(-\frac{\kappa_{H,m}}{2} \sum\limits^m_{i=k} |x_i|^2u^{2H}_i\Big)\, d\overline{u}_k\, d\overline{x}_k\\
&\leq c_7\, (b-a)^{\frac{k}{2}(1-Hd) }\, n^{(m-k)\frac{Hd-1}{2}}  \sum_{\mathcal{S}_k} \int_{O_{m,k}} u_k^{-Hd-H\alpha_{k-1}}u_{k+1}^{-Hd-H(a_{k}+p_{k+1}\alpha_{k+1})}  \\
&\qquad\qquad\qquad \times u_m^{-Hd-H(\overline{p}_{m-1}\alpha_{m-1}+\alpha_m)} \prod^{m-1}_{i=k+2} u_i^{-Hd-H(\overline{p}_{i-1}\alpha_{i-1}+p_i\alpha_i)}\, d\overline{u}_k,
\end{align*}
where $\alpha_i$s are constants in $[0,\beta]$, $\mathcal{S}_k=\big\{p_i, \overline{p}_i:  p_i\in \{0,1\}, p_i+\overline{p}_i=1,\, i=k+1,\dots,m-1\big\}$, $d\overline{u}_k=du_k\cdots du_{m}$, $d\overline{x}_k=dx_k\cdots dx_{m}$ and $O_{m,k}=\{ (u_k, \dots, u_m): 0<u_i \le n(b-a),    i=k,\dots,m\}$.

For any $\varepsilon\in (0,\frac{1-Hd}{2H})$, let
\[
\alpha_{k-1}=\begin{cases}
\frac{1-Hd}{H}-\varepsilon   & \text{if } 1-Hd\leq H\beta; \\ \\
\beta-\varepsilon & \text{if } H\beta<1-Hd<2H\beta,
\end{cases}
\]
and $\alpha_i=\frac{1-Hd}{2H}-\varepsilon$ for $i=k,\cdots,m$. Then 
\begin{align*}
|I^n_{m,k-1}-I^n_{m,k}|
&\leq c_8\, (b-a)^{\frac{k}{2}(1-Hd) }\, n^{(m-k)\frac{Hd-1}{2}} (nb-na)^{\frac{(m-k+1)(1-Hd)}{2}-H\alpha_{k-1}+(m-k+1)H\varepsilon}\\
&\leq c_8\,  (b-a)^{\frac{m(1-Hd)}{2}-\gamma_0+(m-k+2)H\varepsilon}n^{-\gamma_0+(m-k+2)H\varepsilon}.
\end{align*}
Choosing $\varepsilon$ small enough gives
\begin{align*}
|I^n_{m,k-1}-I^n_{m,k}|
&\leq c_9\,  (b-a)^{\frac{m(1-Hd)}{2}-\gamma}n^{-\gamma}.
\end{align*}

\noindent
{\bf Part II}: We consider the case $\beta\in (1,2]$. Suppose that $k$ is odd. Using the similar notation and arguments as in {\bf Part I}, we can show that  
\begin{align*}
|I^n_{m,k-1}-I^n_{m,k}|
&\leq c_{10}\, n^{m\frac{Hd-1}{2}}\int_{\R^{md}}\int_{O_m} \Big(\prod^m_{i=k+1}  |\widehat{f}(x_i-x_{i+1})-\widehat{f}(0)|\Big) \big|\widehat{f}(x_k-x_{k+1})-\widehat{f}(-x_{k+1})\big| \\
&\qquad\qquad\times \Big(\prod^{\frac{k-1}{2}}_{j=1}|\widehat{f}(x_{2j})-\widehat{f}(0)|^2\Big)\, \exp\Big(-\frac{\kappa_{H,m}}{2} \sum\limits^m_{i=1} |x_i|^2u^{2H}_i\Big)\, du\, dx\\
&\leq c_{11}\, n^{m\frac{Hd-1}{2}}\int_{\R^{md}}\int_{O_m}  (|x_{k}|^{\alpha_k}+|x_{k}|^{\theta} |x_{k+1}|^{\alpha_k-\theta}) |x_m|^{\alpha_m}\prod^{m-1}_{i=k+1}  (|x_i|^{\alpha_i}+|x_{i+1}|^{\alpha_i})  \\
&\qquad\qquad\times \Big(\prod^{\frac{k-1}{2}}_{j=1}|\widehat{f}(x_{2j})-\widehat{f}(0)|^2\Big)\exp\Big(-\frac{\kappa_{H,m}}{2} \sum\limits^m_{i=1} |x_i|^2u^{2H}_i\Big)\, du\, dx,
\end{align*}
where we use Lemma \ref{f} in the last inequality with $\alpha_i$s being constants in $[\theta,\beta]$ and $\theta\in(0,1]$. 

Using similar arguments as in {\bf Part I}, we can obtain that 
\begin{align*}
&n^{m\frac{Hd-1}{2}}\int_{\R^{md}}\int_{O_m}  |x_{k}|^{\alpha_k}  |x_m|^{\alpha_m} \prod^{m-1}_{i=k+1}  (|x_i|^{\alpha_i}+|x_{i+1}|^{\alpha_i})  \\
&\qquad\qquad\qquad\times \Big(\prod^{\frac{k-1}{2}}_{j=1}|\widehat{f}(x_{2j})-\widehat{f}(0)|^2\Big)\exp\Big(-\frac{\kappa_{H,m}}{2} \sum\limits^m_{i=1} |x_i|^2u^{2H}_i\Big)\, du\, dx
\end{align*}
is less than a constant multiple of $(b-a)^{\frac{m(1-Hd)}{2}}n^{-\gamma}$. 

So it suffices to estimate
\begin{align*}
&\text{II}:=n^{m\frac{Hd-1}{2}}\int_{\R^{md}}\int_{O_m}  |x_{k}|^{\theta} |x_{k+1}|^{\alpha_k-\theta} |x_m|^{\alpha_m} \prod^{m-1}_{i=k+1}  (|x_i|^{\alpha_i}+|x_{i+1}|^{\alpha_i})  \\
&\qquad\qquad\qquad\times \Big(\prod^{\frac{k-1}{2}}_{j=1}|\widehat{f}(x_{2j})-\widehat{f}(0)|^2\Big)\exp\Big(-\frac{\kappa_{H,m}}{2} \sum\limits^m_{i=1} |x_i|^2u^{2H}_i\Big)\, du\, dx.
\end{align*}
Integrating with respect to $x_i$s and $u_i$s with $i\leq k-1$ gives
\begin{align*}
\text{II} &\leq c_{12}\, (b-a)^{\frac{k-1}{2}(1-Hd) }\, n^{(m-k+1)\frac{Hd-1}{2}}\int_{\R^{(m-k+1)d}}\int_{O_{m,k}} |x_{k}|^{\theta} |x_{k+1}|^{\alpha_k-\theta} |x_m|^{\alpha_m}\\
&\qquad\qquad\qquad \times \prod^{m-1}_{i=k+1}  (|x_i|^{\alpha_i}+|x_{i+1}|^{\alpha_i}) \exp\Big(-\frac{\kappa_{H,m}}{2} \sum\limits^m_{i=k} |x_i|^2u^{2H}_i\Big)\, d\overline{u}_k\, d\overline{x}_k\\
&=c_{12}\, (b-a)^{\frac{k-1}{2}(1-Hd) }\, n^{(m-k+1)\frac{Hd-1}{2}}\sum_{\mathcal{S}_k}\int_{\R^{(m-k+1)d}}\int_{O_{m,k}}  |x_{k}|^{\theta} |x_{k+1}|^{\alpha_k-\theta} |x_m|^{\alpha_m} \\
&\qquad\qquad\qquad\times \prod^{m-1}_{i=k+1}  (|x_i|^{p_i\alpha_i}|x_{i+1}|^{\overline{p}_i\alpha_i}) \exp\Big(-\frac{\kappa_{H,m}}{2} \sum\limits^m_{i=k} |x_i|^2u^{2H}_i\Big)\, d\overline{u}_k\, d\overline{x}_k.
\end{align*}

For any $\varepsilon\in (0,\frac{1-Hd}{2H})$, let
\[
\alpha_{k}=\begin{cases}
\frac{1-Hd}{H}-\varepsilon   & \text{if } 1-Hd\leq H\beta; \\ \\
\beta-\varepsilon & \text{if } H\beta<1-Hd<2H\beta,
\end{cases}
\] 
$\theta=\frac{\gamma_0}{H}$ and $\alpha_i=\frac{1-Hd}{2H}-\varepsilon$ for $i=k+1,\cdots,m$. According to the definition of $\gamma_0$ in (\ref{gamma0}), we see that $\theta\in (0,1]$ and $\alpha_k\in(\theta, \beta]$.  

With these choices of $\theta$ and $\alpha_i$s, integrating with respect to $x_i$s and $u_i$s with $i\geq k$ gives 
\begin{align*}
&|I^n_{m,k-1}-I^n_{m,k}|\leq c_{13}\, (b-a)^{\frac{k-1}{2}(1-Hd) }\, n^{(m-k+1)\frac{Hd-1}{2}} (nb-na)^{(m-k+1)(1-Hd)-H\sum\limits^m_{i=k}\alpha_i}.
\end{align*}

Note that
\begin{align*}
(m-k+1)\frac{Hd-1}{2}+(m-k+1)(1-Hd)-H\sum^m_{i=k}\alpha_i=-\gamma_0+ (m-k+1)H\varepsilon.
\end{align*}
Therefore, choosing $\varepsilon$ small enough, we can obtain that  $\text{II}$ is less than a constant multiple of $ (b-a)^{\frac{m(1-Hd)}{2}-\gamma} n^{-\gamma} $ and thus
\[
|I^n_{m,k-1}-I^n_{m,k}|\leq c_{14}\, (b-a)^{\frac{m(1-Hd)}{2}} n^{-\gamma}.
\]

Suppose now that $k$ is even. Using the similar notation and arguments as above, we can show that $|I^n_{m,k-1}-I^n_{m,k}|$ is less than a constant multiple of
\begin{align*}
&|I^n_{m,k-1}-I^n_{m,k}|\\
&\leq c_{15}\, n^{m\frac{Hd-1}{2}}\int_{\R^{md}}\int_{O_m} \Big(\prod^{\frac{k-2}{2}}_{j=1}|\widehat{f}(x_{2j})-\widehat{f}(0)|^2\Big) \big|\widehat{f}(-x_k)-\widehat{f}(0)\big|\big|\widehat{f}(x_k-x_{k+1})-\widehat{f}(x_k)\big| \\
&\qquad\qquad\qquad\qquad\times \prod^m_{i=k+1}  |\widehat{f}(x_i-x_{i+1})-\widehat{f}(0)|\, \exp\Big(-\frac{\kappa_{H,m}}{2} \sum\limits^m_{i=1} |x_i|^2u^{2H}_i\Big)\, du\, dx\\
&\leq c_{16}\, (b-a)^{\frac{k}{2}(1-Hd) }\, n^{(m-k)\frac{Hd-1}{2}} \int_{\R^{(m-k+1)d}}\int_{O_{m,k}}  |x_k|^{\alpha_{k-1}}(|x_{k+1}|^{\alpha_k}+|x_{k+1}|^{\theta}|x_k|^{\alpha_k-\theta})|x_m|^{\alpha_m}\\
&\qquad\qquad\qquad\times \Big(\prod^{m-1}_{i=k+1}  (|x_i|^{\alpha_i}+|x_{i+1}|^{\alpha_i})\Big)\exp\Big(-\frac{\kappa_{H,m}}{2} \sum\limits^m_{i=k} |x_i|^2u^{2H}_i\Big)\, d\overline{u}_k\, d\overline{x}_k.
\end{align*}

Using similar arguments as in {\bf Part I}, we can obtain that 
\begin{align*}
&(b-a)^{\frac{k}{2}(1-Hd) }\, n^{(m-k)\frac{Hd-1}{2}} \int_{\R^{(m-k+1)d}}\int_{O_{m,k}}  |x_k|^{\alpha_{k-1}}|x_{k+1}|^{\alpha_k}|x_m|^{\alpha_m}\\
&\qquad\qquad\qquad\times \Big(\prod^{m-1}_{i=k+1}  (|x_i|^{\alpha_i}+|x_{i+1}|^{\alpha_i})\Big)\exp\Big(-\frac{\kappa_{H,m}}{2} \sum\limits^m_{i=k} |x_i|^2u^{2H}_i\Big)\, d\overline{u}_k\, d\overline{x}_k.
\end{align*}
is less than a constant multiple of $(b-a)^{\frac{m(1-Hd)}{2}-\gamma}n^{-\gamma}$. 

So we only need to estimate
\begin{align*}
&\text{III}:= (b-a)^{\frac{k}{2}(1-Hd) }\, n^{(m-k)\frac{Hd-1}{2}} \int_{\R^{(m-k+1)d}}\int_{O_{m,k}}  |x_k|^{\alpha_{k-1}+\alpha_k-\theta}|x_{k+1}|^{\theta} |x_m|^{\alpha_m}\\
&\qquad\qquad\qquad\times \Big(\prod^{m-1}_{i=k+1}  (|x_i|^{\alpha_i}+|x_{i+1}|^{\alpha_i})\Big)\exp\Big(-\frac{\kappa_{H,m}}{2} \sum\limits^m_{i=k} |x_i|^2u^{2H}_i\Big)\, d\overline{u}_k\, d\overline{x}_k.
\end{align*}
For any $\varepsilon\in (0,\frac{1-Hd}{2H})$, let
\[
\alpha_{k-1}=\begin{cases}
\frac{1-Hd}{H}-\varepsilon   & \text{if } 1-Hd\leq H\beta; \\ \\
\beta-\varepsilon & \text{if } H\beta<1-Hd<2H\beta,
\end{cases}
\]
$\theta=\frac{\gamma_0}{H}$ and $\alpha_i=\frac{1-Hd}{2H}-\varepsilon$ for $i=k,\cdots, m$. It is easy to see that $\alpha_{k-1}+\alpha_k-\theta\in (0,\frac{1-Hd}{H})$ and $\alpha_{k+1}+\theta\in (0,\frac{1-Hd}{H})$. 

With these choices of $\theta$ and $\alpha_i$s, we can easily get
\begin{align*}
&|I^n_{m,k-1}-I^n_{m,k}|\\
&\leq c_{17}\, (b-a)^{\frac{k}{2}(1-Hd) }\, n^{(m-k)\frac{Hd-1}{2}} (nb-na)^{(m-k+1)(1-Hd)-H\alpha_{k-1}-\frac{(m-k+1)(1-Hd)}{2}+(m-k+1)H\varepsilon}\\
&\leq c_{17}\,  (b-a)^{\frac{m(1-Hd)}{2}-\gamma_0+(m-k+2)H\varepsilon}n^{-\gamma_0+(m-k+2)H\varepsilon}.
\end{align*}
Choosing $\varepsilon$ small enough gives
\begin{align*}
|I^n_{m,k-1}-I^n_{m,k}|
&\leq c_{18}\,  (b-a)^{\frac{m(1-Hd)}{2}-\gamma}n^{-\gamma}.
\end{align*}
This completes the proof.
\end{proof}

\begin{proposition} \label{chain2} Suppose that $\frac{1}{2\beta+d}=H$ and $f\in \Theta^{\beta}$ for some $\beta\in \{1, 2\}$. For $k=1,2,\dots,m$, there exists a positive constant $c$ such that
\[
|I^n_{m,k-1}-I^n_{m,k}|\leq c\, (b-a)^{\frac{m(1-Hd)}{2}} (\ln n)^{-\frac{1}{2}}.
\]
\end{proposition}

\begin{proof}  We divide the proof into two parts. 

\noindent
{\bf Part I}: We consider the case $\beta=1$. In this case $1-Hd=2H$. Suppose that  $k$ is odd. Applying the local nondeterminism property  (\ref{lndp}) and making the change of variable $u_1=s_1$, $u_i =s_i -s_{i-1}$, for $2\le i\le m$, we can show that $|I^n_{m,k-1}-I^n_{m,k}|$ is less than a constant multiple of
\begin{align*}
&(\ln n)^{-\frac{m}{2}}n^{m\frac{Hd-1}{2}}\int_{\R^{md}}\int_{O_m} \Big(\prod^m_{i=k+1}  |\widehat{f}(x_i-x_{i+1})-\widehat{f}(0)|\Big)\\
&\qquad\quad\times \big|\widehat{f}(x_k-x_{k+1})-\widehat{f}(-x_{k+1})\big|\Big(\prod^{\frac{k-1}{2}}_{j=1}|\widehat{f}(x_{2j})-\widehat{f}(0)|^2\Big)\, \exp\Big(-\frac{\kappa_{H,m}}{2} \sum\limits^m_{i=1} |x_i|^2u^{2H}_i\Big)\, du\, dx,
\end{align*}
where $O_{m}=\big\{ (u_k, \dots, u_m):   0<u_i \le n(b-a),    i=1,\dots,m\big\}$.

Applying Lemmas \ref{f} and \ref{int} and integrating with respect to $x_i$s and $u_i$s for $i\leq k-1$, we can obtain that
\begin{align*}
&|I^n_{m,k-1}-I^n_{m,k}|\\
&\qquad\leq c_1\, (b-a)^{\frac{k-1}{2}(1-Hd) }\, (\ln n)^{-\frac{m-k+1}{2}}\, n^{(m-k+1)\frac{Hd-1}{2}}\int_{\R^{(m-k+1)d}}\int_{O_{m,k}} (|x_{k}|\wedge 1)(|x_m|\wedge 1)\\
&\qquad\qquad\times  \prod^{m-1}_{i=k+1}  \big(|x_i|\wedge1+|x_{i+1}|\wedge 1\big) \exp\Big(-\frac{\kappa_{H,m}}{2} \sum\limits^m_{i=k} |x_i|^2u^{2H}_i\Big)\, d\overline{u}_k\, d\overline{x}_k\\
&\qquad\leq c_1\, (b-a)^{\frac{k-1}{2}(1-Hd) } (\ln n)^{-\frac{m-k+1}{2}} \, n^{(m-k+1)\frac{Hd-1}{2}}\sum_{\overline{\mathcal{S}}_k}\int_{\R^{(m-k+1)d}}\int_{O_{m,k}} (|x_{k}|\wedge 1)\\
&\qquad\qquad\times \Big(\prod^{m}_{i=k+1}  (|x_i|\wedge 1)^{p_i+\overline{p}_{i-1}}\Big) \exp\Big(-\frac{\kappa_{H,m}}{2} \sum\limits^m_{i=k} |x_i|^2u^{2H}_i\Big)\, d\overline{u}_k\, d\overline{x}_k,
\end{align*}
where $O_{m,k}=\big\{ (u_k, \dots, u_m): 0<u_i \le n(b-a),    i=k,\dots,m\big\}$, $d\overline{u}_k=du_k\cdots du_{m}$, $d\overline{x}_k=dx_k\cdots dx_{m}$ and $
\overline{\mathcal{S}}_k=\big\{p_k=p_m=1, p_i\in \{0,1\}, p_i+\overline{p}_i=1,\, i=k,\dots,m\big\}$.

By Lemma \ref{int}, 
\begin{align*}
&|I^n_{m,k-1}-I^n_{m,k}|\\
&\leq c_2\, (b-a)^{\frac{k-1}{2}(1-Hd) }(\ln n)^{-\frac{m-k+1}{2}} \, n^{(m-k+1)\frac{Hd-1}{2}} (nb-na)^{(m-k+1)(1-Hd-H)} (\ln n)^{|A|}\\
&=c_2\, (b-a)^{\frac{m(1-Hd)}{2}}\, (\ln n)^{-\frac{m-k+1}{2}+|A|},
\end{align*}
where $A=\{i=k+1,\cdots, m: p_i+\overline{p}_{i-1}=2\}$ and $|A|$ is the number of elements in $A$.

Note that $\overline{p}_{k}=0$. Hence $|A|\leq \frac{(m-k-1)\vee 0}{2}$ and 
\[
|I^n_{m,k-1}-I^n_{m,k}|\leq c_3\, (b-a)^{\frac{m(1-Hd)}{2}}\, (\ln n)^{-\frac{1}{2}}.
\]

\medskip
 
Suppose now that $k$ is even. Using similar notation and arguments as above, we can obtain that
\begin{align*}
&|I^n_{m,k-1}-I^n_{m,k}|\\
&\leq c_4\, (\ln n)^{-\frac{m}{2}}  n^{m\frac{Hd-1}{2}}\int_{\R^{md}}\int_{O_m} \Big(\prod^{\frac{k-2}{2}}_{j=1}|\widehat{f}(x_{2j})-\widehat{f}(0)|^2\Big) \big|\widehat{f}(-x_k)-\widehat{f}(0)\big|\big|\widehat{f}(x_k-x_{k+1})-\widehat{f}(x_k)\big| \\
&\qquad\qquad\qquad\times \prod^m_{i=k+1}  |\widehat{f}(x_i-x_{i+1})-\widehat{f}(0)|\, \exp\Big(-\frac{\kappa_{H,m}}{2} \sum\limits^m_{i=1} |x_i|^2u^{2H}_i\Big)\, du\, dx\\
&\leq c_5\, (b-a)^{\frac{k}{2}(1-Hd)}(\ln n)^{-\frac{m-k+2}{2}} \, n^{(m-k)\frac{Hd-1}{2}}  \sum_{\widetilde{\mathcal{S}}_k} \int_{\R^{(m-k+1)d}}\int_{O_{m,k}}  (|x_k|\wedge 1) \prod^{m}_{i=k+1}  (|x_i|\wedge 1)^{p_i+\overline{p}_{i-1}} \\
&\qquad\qquad\qquad \times   \exp\Big(-\frac{\kappa_{H,m}}{2} \sum\limits^m_{i=k} |x_i|^2u^{2H}_i\Big)\, d\overline{u}_k\, d\overline{x}_k,
\end{align*}
where $\widetilde{\mathcal{S}}_k=\big\{\overline{p}_k=1, p_m=1, p_i\in \{0,1\}, p_i+\overline{p}_i=1,\, i=k,\dots,m\big\}$. 

By Lemma \ref{int},
\begin{align*}
&|I^n_{m,k-1}-I^n_{m,k}|\\
&\leq c_6\, (b-a)^{\frac{k}{2}(1-Hd) } (\ln n)^{-\frac{m-k+2}{2}} \, n^{(m-k)\frac{Hd-1}{2}} (nb-na)^{(m-k+1)(1-Hd)-(m-k+2)H}(\ln n)^{|A|},
\end{align*}
where $A=\{i=k+1,\cdots, m: p_i+\overline{p}_{i-1}=2\}$.

Clearly $|A|\leq \frac{m-k}{2}$. Hence
\begin{align*}
|I^n_{m,k-1}-I^n_{m,k}|\leq c_7\, (b-a)^{\frac{m(1-Hd)}{2}}\, (\ln n)^{-\frac{1}{2}}.
\end{align*}

\noindent
{\bf Part II}: We consider the case $\beta=2$. In this case $1-Hd=4H$. Suppose that $k$ is odd. Using the similar notation and arguments as in {\bf Part I}, we can show that $|I^n_{m,k-1}-I^n_{m,k}|$ is less than a constant multiple of
\begin{align*}
&(\ln n)^{-\frac{m}{2}} \, n^{m\frac{Hd-1}{2}}\int_{\R^{md}}\int_{O_m} \Big(\prod^m_{i=k+1}  |\widehat{f}(x_i-x_{i+1})-\widehat{f}(0)|\Big)\\
&\qquad\quad\times \big|\widehat{f}(x_k-x_{k+1})-\widehat{f}(-x_{k+1})\big|\Big(\prod^{\frac{k-1}{2}}_{j=1}|\widehat{f}(x_{2j})-\widehat{f}(0)|^2\Big)\, \exp\Big(-\frac{\kappa_{H,m}}{2} \sum\limits^m_{i=1} |x_i|^2u^{2H}_i\Big)\, du\, dx.
\end{align*}
By Lemma \ref{f}, $|I^n_{m,k-1}-I^n_{m,k}|$ is less than a constant multiple of
\begin{align*}
&(\ln n)^{-\frac{m}{2}} n^{m\frac{Hd-1}{2}}\int_{\R^{md}}\int_{O_m}  \Big((|x_{k}|^2+|x_{k}||x_{k+1}|)\wedge 1\Big)  \big(|x_m|^2\wedge 1\big)\prod^{m-1}_{i=k+1} \Big(|x_i|^2\wedge 1+|x_{i+1}|^2\wedge 1\Big)  \\
&\qquad\qquad\qquad\times \Big(\prod^{\frac{k-1}{2}}_{j=1}|\widehat{f}(x_{2j})-\widehat{f}(0)|^2\Big)\exp\Big(-\frac{\kappa_{H,m}}{2} \sum\limits^m_{i=1} |x_i|^2u^{2H}_i\Big)\, du\, dx.
\end{align*} 

Using similar arguments as in {\bf Part I}, we can obtain that 
\begin{align*}
&(\ln n)^{-\frac{m}{2}} n^{m\frac{Hd-1}{2}}\int_{\R^{md}}\int_{O_m}  (|x_{k}|^2\wedge 1)  (|x_m|^2 \wedge 1) \prod^{m-1}_{i=k+1}  \Big(|x_i|^2\wedge 1+|x_{i+1}|^2\wedge 1\Big)  \\
&\qquad\qquad\qquad\times \Big(\prod^{\frac{k-1}{2}}_{j=1}|\widehat{f}(x_{2j})-\widehat{f}(0)|^2\Big)\exp\Big(-\frac{\kappa_{H,m}}{2} \sum\limits^m_{i=1} |x_i|^2u^{2H}_i\Big)\, du\, dx
\end{align*}
is less than a constant multiple of $(b-a)^{\frac{m(1-Hd)}{2}}(\ln n)^{-\frac{1}{2}}$. 

So it suffices to estimate
\begin{align*}
&\text{II}:=(\ln n)^{-\frac{m}{2}} n^{m\frac{Hd-1}{2}}\int_{\R^{md}}\int_{O_m}  |x_{k}||x_{k+1}| (|x_m|^2\wedge 1) \prod^{m-1}_{i=k+1}  \Big(|x_i|^2\wedge 1+|x_{i+1}|^2\wedge 1\Big)  \\
&\qquad\qquad\qquad\times \Big(\prod^{\frac{k-1}{2}}_{j=1}|\widehat{f}(x_{2j})-\widehat{f}(0)|^2\Big)\exp\Big(-\frac{\kappa_{H,m}}{2} \sum\limits^m_{i=1} |x_i|^2u^{2H}_i\Big)\, du\, dx.
\end{align*}
Applying Lemma \ref{int} and integrating with respect to $x_i$s and $u_i$s with $i\leq k-1$ give
\begin{align*}
\text{II} &\leq c_8\, (b-a)^{\frac{k-1}{2}(1-Hd) } (\ln n)^{-\frac{m-k+1}{2}}  \, n^{(m-k+1)\frac{Hd-1}{2}}\sum_{\overline{\overline{\mathcal{S}}}_k}\int_{\R^{(m-k+1)d}}\int_{O_{m,k}}  |x_{k}||x_{k+1}| \\
&\qquad\qquad\times (|x_{k+1}|^{2}\wedge 1)^{p_{k+1}} \Big( \prod^{m}_{i=k+2}  (|x_i|^{2}\wedge 1)^{p_i+\overline{p}_{i-1}} \Big) \exp\Big(-\frac{\kappa_{H,m}}{2} \sum\limits^m_{i=k} |x_i|^2u^{2H}_i\Big)\, d\overline{u}_k\, d\overline{x}_k,
\end{align*}
where $\overline{\overline{\mathcal{S}}}_k=\big\{p_m=1, p_i\in \{0,1\}, p_i+\overline{p}_i=1,\, i=k+1,\dots,m\big\}$.

By Lemma \ref{int}, 
\begin{align*}
\text{II} &\leq c_8\, (b-a)^{\frac{k-1}{2}(1-Hd) } (\ln n)^{-\frac{m-k+1}{2}}  \, n^{(m-k+1)\frac{Hd-1}{2}} (nb-na)^{(m-k+1)(1-Hd)-2(m-k+1)H} (\ln n)^{|B|},
\end{align*}
where $B=\{i=k+2,\cdots, m: p_i+\overline{p}_{i-1}=2\}$ and $|B|$ is the number of elements in $B$.

Note that  $|B|\leq \frac{(m-k-1)\vee 0}{2}$. Hence
\begin{align*}
\text{II}  \leq c_9\, (b-a)^{\frac{m(1-Hd)}{2}}\, (\ln n)^{-\frac{1}{2}}.
\end{align*}

Suppose now that $k$ is even. Using similar notation and arguments as above,
\begin{align*}
&|I^n_{m,k-1}-I^n_{m,k}|\\
&\leq c_{10}\, (\ln n)^{-\frac{m}{2}}\, n^{m\frac{Hd-1}{2}}\int_{\R^{md}}\int_{O_m} \Big(\prod^{\frac{k-2}{2}}_{j=1}|\widehat{f}(x_{2j})-\widehat{f}(0)|^2\Big) \big|\widehat{f}(-x_k)-\widehat{f}(0)\big|\big|\widehat{f}(x_k-x_{k+1})-\widehat{f}(x_k)\big| \\
&\qquad\qquad\qquad\qquad\times \prod^m_{i=k+1}  |\widehat{f}(x_i-x_{i+1})-\widehat{f}(0)|\, \exp\Big(-\frac{\kappa_{H,m}}{2} \sum\limits^m_{i=1} |x_i|^2u^{2H}_i\Big)\, du\, dx\\
&\leq c_{11}\, (b-a)^{\frac{k}{2}(1-Hd) }(\ln n)^{-\frac{m-k+2}{2}} \, n^{(m-k)\frac{Hd-1}{2}} \int_{\R^{(m-k+1)d}}\int_{O_{m,k}} (|x_k|^2\wedge 1)\Big((|x_{k+1}|^2+|x_{k+1}||x_k|)\wedge 1 \Big)\\
&\qquad\qquad\qquad\times (|x_m|^2 \wedge 1) \Big(\prod^{m-1}_{i=k+1}  (|x_i|^2\wedge 1+|x_{i+1}|^2\wedge 1)\Big)\exp\Big(-\frac{\kappa_{H,m}}{2} \sum\limits^m_{i=k} |x_i|^2u^{2H}_i\Big)\, d\overline{u}_k\, d\overline{x}_k.
\end{align*}

Using similar arguments as in {\bf Part I}, we can obtain that 
\begin{align*}
&(b-a)^{\frac{k}{2}(1-Hd) }(\ln n)^{-\frac{m-k+2}{2}} \, n^{(m-k)\frac{Hd-1}{2}} \int_{\R^{(m-k+1)d}}\int_{O_{m,k}} (|x_k|^2\wedge 1)(|x_{k+1}|^2\wedge 1) \\
&\qquad\qquad\times (|x_m|^2\wedge 1) \Big(\prod^{m-1}_{i=k+1}  (|x_i|^2\wedge 1+|x_{i+1}|^2\wedge 1)\Big)\exp\Big(-\frac{\kappa_{H,m}}{2} \sum\limits^m_{i=k} |x_i|^2u^{2H}_i\Big)\, d\overline{u}_k\, d\overline{x}_k.
\end{align*}
is less than a constant multiple of $(b-a)^{\frac{m(1-Hd)}{2}}(\ln n)^{-\frac{1}{2}}$. 

So it suffices to estimate
\begin{align*}
&\text{III}:= (b-a)^{\frac{k}{2}(1-Hd) } (\ln n)^{-\frac{m-k+2}{2}}  \, n^{(m-k)\frac{Hd-1}{2}} \int_{\R^{(m-k+1)d}}\int_{O_{m,k}}  (|x_k|^2\wedge 1) |x_k| |x_{k+1}| (|x_m|^2\wedge 1)\\
&\qquad\qquad\qquad\times \Big(\prod^{m-1}_{i=k+1}  (|x_i|^2\wedge 1+|x_{i+1}|^2\wedge 1)\Big)\exp\Big(-\frac{\kappa_{H,m}}{2} \sum\limits^m_{i=k} |x_i|^2u^{2H}_i\Big)\, d\overline{u}_k\, d\overline{x}_k.
\end{align*}

Note that
\begin{align*}
&\text{III}\leq (b-a)^{\frac{k}{2}(1-Hd) } (\ln n)^{-\frac{m-k+2}{2}}  \, n^{(m-k)\frac{Hd-1}{2}} \sum_{\overline{\overline{\mathcal{S}}}_k}\int_{\R^{(m-k+1)d}}\int_{O_{m,k}}  |x_k|^3 |x_{k+1}| \\
&\qquad\qquad\qquad\times (|x_{k+1}|^2\wedge 1)^{p_{k+1}} \Big(\prod^{m}_{i=k+2}  (|x_i|^2\wedge 1)^{p_{i}+\overline{p}_{i-1}}\Big)\exp\Big(-\frac{\kappa_{H,m}}{2} \sum\limits^m_{i=k} |x_i|^2u^{2H}_i\Big)\, d\overline{u}_k\, d\overline{x}_k,
\end{align*}
where $\overline{\overline{\mathcal{S}}}_k=\big\{p_m=1, p_i\in \{0,1\}, p_i+\overline{p}_i=1,\, i=k+1,\dots,m\big\}$.

By Lemma \ref{int}, 
\begin{align*}
\text{III}&\leq c_{12}\, (b-a)^{\frac{k}{2}(1-Hd) } (\ln n)^{-\frac{m-k+2}{2}}  \, n^{(m-k)\frac{Hd-1}{2}}  (nb-na)^{(m-k+1)(1-Hd)-2(m+2-k)H}(\ln n)^{|B|},
\end{align*}
where $B=\{i=k+2,\cdots, m: p_i+\overline{p}_{i-1}=2\}$.

Note that  $|B|\leq \frac{(m-k-1)\vee 0}{2}$. Hence
\begin{align*}
\text{III}  \leq c_{13}\, (b-a)^{\frac{m(1-Hd)}{2}} \, (\ln n)^{-\frac{1}{2}}.
\end{align*}
This completes the proof.
\end{proof}

\bigskip

\section{Proofs of the main results}

Both proofs of Theorems \ref{th1} and \ref{th2} will be done in two steps. We first
show the tightness and then establish the convergence of moments. Tightness will be deduced from the following proposition.

\begin{proposition} \label{tight}
For any $0\leq a<b \leq t$ and any integer $m\geq 1$, there exists a positive constant $C$ depending on $H$, $m$, $d$, $\beta$, $t$ and $f$ such that
\begin{equation*}
\E\big[(F_{n}(b)-F_{n}(a))^{2m}\big] 
 \leq C
\begin{cases}
(b-a)^{m(1-Hd)-\gamma} &   \text{if} \quad \frac{1}{2\beta+d}<H<\frac{1}{d} \; \text{with}\; \beta\in (0,2];\\ \\
(b-a)^{m(1-Hd)} &   \text{if} \quad  H=\frac{1}{2\beta+d} \; \text{with}\; \beta\in \{1,2\},
\end{cases}
\end{equation*}
where $\gamma$ is the positive constant in Proposition \ref{chain1}.

\end{proposition}
\begin{proof} Note that $\E\big[(F_{n}(b)-F_{n}(a))^{2m}\big] =I^n_{2m,0}$. In the case $\frac{1}{2\beta+d}<H<\frac{1}{H}$, by Proposition \ref{chain1},
\begin{align*} %\label{tight1}
|I^n_{2m,0}-I^n_{2m,2m}|\leq c_1\,  (b-a)^{m(1-Hd)-\gamma}\, n^{-\gamma}.
\end{align*}
In the case $H=\frac{1}{2\beta+d}$, by Proposition \ref{chain2},
\begin{align*} %\label{tight2}
|I^n_{2m,0}-I^n_{2m,2m}|\leq c_2\,  (b-a)^{m(1-Hd)}\, (\ln n)^{-\frac{1}{2}}.
\end{align*}

So it suffices to estimate $I^n_{2m,2m}$. In the first case,
\begin{align*}   
|I^n_{2m,2m}|
&\leq c_3\, n^{m(Hd-1)}\int_{\R^{2md}}\int_{O_{2m}}\Big(\prod^m_{j=1}|\widehat{f}(x_{2j})-\widehat{f}(0)|^2\Big)\, \exp\Big(-\frac{\kappa_{H,2m}}{2}\sum\limits^{2m}_{i=1} |x_i|^2u^{2H}_i\Big)\, du\, dx\\ 
&\leq c_4\, (b-a)^{m(1-Hd)}\Big(\int_{\R^d}|\widehat{f}(x)-\widehat{f}(0)|^2|x|^{-\frac{1}{H}}\, dx\Big)^m,
\end{align*}
where $O_{2m}=\big\{ (u_1, \dots, u_{2m}):   0<u_i \le n(b-a),    i=1,\dots, 2m\big\}$.

In the second case,
\begin{align*}  
&|I^n_{2m,2m}|\\
&\leq c_5\, (\ln n)^{-m}\, n^{m(Hd-1)}\int_{\R^{2md}}\int_{O_{2m}}\Big(\prod^m_{j=1}|\widehat{f}(x_{2j})-\widehat{f}(0)|^2\Big)\, \exp\Big(-\frac{\kappa_{H,2m}}{2}\sum\limits^{2m}_{i=1} |x_i|^2u^{2H}_i\Big)\, du\, dx\\ 
&\leq c_6\, (b-a)^{m(1-Hd)}\Big(\frac{1}{\ln n}\int^{n(b-a)}_0 \int_{\R^d}|\widehat{f}(x)-\widehat{f}(0)|^2 e^{-\frac{1}{2}|x|^2u^{2H}}\, dx\, du\Big)^m\\
&\leq c_7\, (b-a)^{m(1-Hd)},
\end{align*}
where we use Lemmas \ref{f} and \ref{int} in the last inequality.

Combining all these estimates gives the desired result.
\end{proof}

\bigskip

Next we shall prove the convergence of all finite dimensional distributions. That is, we shall prove that the moments of $F_n(t)$ converge to the
corresponding ones of $W(L_t(\lambda))$.

Fix a finite number of disjoint intervals $(a_i, b_i]$ with
$i=1,\dots ,N$ and
 $b_i\le a_{i+1}$.  Let $\mathbf{m}=(m_1, \dots, m_N)$ be a fixed multi-index with $m_i\in\N$ for $i=1,\dots ,N$. Set $\sum\limits_{i=1}^N m_i=|\mathbf{m}|$ and  $\prod\limits_{i=1}^N m_i!=\mathbf{m}!$.
We need to consider the following sequence of random variables
 \[
        G_n=\prod_{i=1}^N \left( F_n(b_i)- F_n(a_i) \right)^{m_i}
 \]
and compute  $\lim\limits_{n\rightarrow  \infty}  \E(G_n) $. By Proposition \ref{tight}, we can assume that $b_i<a_{i+1}$ for $i=1,2\cdots, N-1$. Let 
\begin{equation}
D_\mathbf{m}=\big\{s \in \R^{|\mathbf{m}|}: na_{i}<s^i_{1}<\cdots <s^i_{m_i}<nb_{i}, 1\le i\le N\big\}. \label{e.3.1.4}
\end{equation}
Here and in the sequel we denote the coordinates of a point $s\in  \R^{|\mathbf{m}|}$ as
$s= (s^i_j)$, where   $ 1\le i \le   N$ and $1\le j \le m_i $.

For simplicity of notation, we define 
\[
J_0=\big\{(i,j): 1\leq i\leq N, 1\leq j\leq m_i \big\}.
\]
For any $(i_1, j_1)$ and $(i_2,j_2)\in J_0$, we define the following dictionary ordering 
\[
(i_1, j_1)\leq (i_2,j_2)
\]
if $i_1<i_2$ or $i_1=i_2$ and $j_1\leq j_2$. For any $(i,j)$ in $J_0$, under the above ordering, $(i,j)$ is the $(\sum\limits^{i-1}_{k=1}m_k+j)$-th element in $J_0$ and we define $\#(i,j)=\sum\limits^{i-1}_{k=1}m_k+j$.

\begin{proposition} \label{odd} Suppose that at least one of the exponents $m_i$ is odd. Then
\begin{equation*}
\lim\limits_{n\to\infty}\E(G_n)=0.
\end{equation*}
\end{proposition}
\begin{proof}  Using Fourier transform, we see that $\E (G_n)$ is equal to
\begin{align*}
& \frac{\mathbf{m}!}{(2\pi)^{|\mathbf{m}|d}}\,   \left(\frac{\ell_{n,H, d}}{n} \right)^{|\mathbf{m}|} \int_{\R^{|\mathbf{m}|d}}  \int_{D_\mathbf{m}} 
           \Big( \prod^{N}_{i=1}\prod^{m_i}_{j=1} (\widehat{f}(y^i_j)-\widehat{f}(0))\Big) \\
           &\qquad\qquad\times \exp\bigg(-\iota n^H\lambda\cdot \sum^{N}_{i=1}\sum^{m_i}_{j=1} y^i_j-\frac{1}{2}\Var\Big( \sum^{N}_{i=1}\sum^{m_i}_{j=1} y^i_j\cdot X_{s^i_j} \Big)\bigg) \, ds\, dy.
\end{align*}
Making the change of variables $x^i_j=\sum\limits_{(\ell,k)\geq (i,j)} y^{\ell}_k$ for $1\leq i \leq N$ and $1\leq j\leq m_i$,
\begin{align*}
\E (G_n)
&=\frac{\mathbf{m}!}{(2\pi)^{|\mathbf{m}|d}}\,   \left(\frac{\ell_{n,H, d}}{n} \right)^{|\mathbf{m}|} \int_{\R^{|\mathbf{m}|d}}  \int_{D_\mathbf{m}} 
            \prod^{N}_{i=1}\prod^{m_i}_{j=1} (\widehat{f}(x^i_j-x^i_{j+1})-\widehat{f}(0))\\
&\qquad\qquad\times \exp\bigg(-\iota n^H\lambda\cdot x^1_1-\frac{1}{2}\Var\Big( \sum^{N}_{i=1}\sum^{m_i}_{j=1} x^i_j\cdot \big(X_{s^i_j}-X_{s^i_{j-1}}\big) \Big)\bigg) \, ds\, dx,
\end{align*}
where $s^i_0=s^{i-1}_{m_{i-1}}$ for $i=2,\dots,N$ and $s^1_0=0$.

Applying Propositions \ref{chain1} and \ref{chain2},  we  obtain
\begin{align*}
\lim_{n\to\infty}\E (G_n)
&=\frac{\mathbf{m}!}{(2\pi)^{|\mathbf{m}|d}}\,  \lim_{n\to\infty}  \left(\frac{\ell_{n,H, d}}{n} \right)^{|\mathbf{m}|}  \int_{\R^{|\mathbf{m}|d}}  \int_{D_\mathbf{m}} 
           \bigg( \prod_{(i,j)\in J_e} |\widehat{f}(x^i_j)-\widehat{f}(0)|^2\bigg)\, I_{|\mathbf{m}|}\\
&\qquad\qquad\times \exp\bigg(-\iota n^H\lambda\cdot x^1_1-\frac{1}{2}\Var\Big( \sum^{N}_{i=1}\sum^{m_i}_{j=1} x^i_j\cdot \big(X_{s^i_j}-X_{s^i_{j-1}}\big) \Big)\bigg) \, ds\, dx,
\end{align*}
where $J_e=\big\{(i,j)\in J_0: \#(i,j)\; \text{is even}\big\}$ and 
\[
I_{|\mathbf{m}|}=
\begin{cases}
\widehat{f}(x^N_{m_N})-\widehat{f}(0), & \text{if}\; |\mathbf{m}| \; \text{is odd};\\  \\
1, & \text{if}\; |\mathbf{m}| \; \text{is even}.
\end{cases}
\]

It is easy to see that $\lim\limits_{n\to\infty}\E (G_n)=0$ when $|\mathbf{m}|$ is odd. So it suffices to show $\lim\limits_{n\to\infty}\E (G_n)=0$ when $|\mathbf{m}|$ is even. In this case,
\begin{align*}
\limsup_{n\to\infty}|\E (G_n)|
&\leq \frac{\mathbf{m}!}{(2\pi)^{|\mathbf{m}|d}}\,  \limsup_{n\to\infty}  \left(\frac{\ell_{n,H, d}}{n} \right)^{|\mathbf{m}|}  \int_{\R^{|\mathbf{m}|d}}  \int_{D_\mathbf{m}} 
           \Big( \prod_{(i,j)\in J_e} |\widehat{f}(x^i_j)-\widehat{f}(0)|^2\Big)\\
&\qquad\qquad\times \exp\bigg(-\frac{1}{2}\Var\Big( \sum^{N}_{i=1}\sum^{m_i}_{j=1} x^i_j\cdot \big(X_{s^i_j}-X_{s^i_{j-1}}\big) \Big)\bigg) \, ds\, dx.
\end{align*}
Using the local nondeterminism property (\ref{lndp}),
\begin{align*}
\limsup_{n\to\infty}\big|\E (G_n)\big|
&\leq c_1  \limsup_{n\to\infty}  \left(\frac{\ell_{n,H, d}}{n} \right)^{|\mathbf{m}|} \int_{\R^{|\mathbf{m}|d}}  \int_{D_\mathbf{m}} 
           \Big( \prod_{(i,j)\in J_e} |\widehat{f}(x^i_j)-\widehat{f}(0)|^2\Big)\\
&\qquad\qquad\times \exp\bigg(-\frac{\kappa_{H,|\mathbf{m}|}}{2}\sum^{N}_{i=1}\sum^{m_i}_{j=1} |x^i_j|^2(s^i_j-s^i_{j-1})^{2H}\bigg) \, ds\, dx\\
&:= c_1 \limsup_{n\to\infty} \text{II}_n.
\end{align*}

Assume that $m_{\ell}$ is the first odd exponent.  Integrating with respect to proper $x^i_j$s and $s^i_j$s gives
\begin{align*}
\text{II}_n &\leq c_2\, \left(\frac{\ell_{n,H, d}}{n}\right)^2 \sup_{s^{\ell}_{m_{\ell}-1}\in (na_\ell,nb_{\ell}]} \int_{\R^{d}}  \int^{nb_{\ell}}_{s^{\ell}_{m_{\ell}-1}}\int^{nb_{\ell+1}}_{na_{\ell+1}}
           |\widehat{f}(x^{\ell+1}_1)-\widehat{f}(0)|^2(s^{\ell}_{m_{\ell}}-s^{\ell}_{m_{\ell}-1})^{-Hd}\\
&\qquad\qquad\times   \exp\Big(-\frac{\kappa_{H,|\mathbf{m}|}}{2} |x^{\ell+1}_1|^2(s^{\ell+1}_1-s^{\ell}_{m_{\ell}})^{2H}\Big) \, ds^{\ell+1}_1\, ds^{\ell}_{m_{\ell}}\, dx^{\ell+1}_1.
\end{align*}
Note that $|\widehat{f}(x)-\widehat{f}(0)|\leq c_{\beta}(|x|^{\beta}\wedge 1)$. Then
\begin{align*}
\text{II}_n &\leq c_3\, \left(\frac{\ell_{n,H, d}}{n}\right)^2 n^{1-Hd}  \int_{\R^{d}}   \int^{nb_{\ell+1}}_{n a_{\ell+1}}
           |x^{\ell+1}_1|^{2\beta} \exp\Big(-\frac{\kappa_{H,|\mathbf{m}|}}{2} |x^{\ell+1}_1|^2(s^{\ell+1}_1-nb_{\ell})^{2H}\Big) \, ds^{\ell+1}_1\,  dx^{\ell+1}_1  \nonumber \\  
&\leq c_4\,  \left(\frac{\ell_{n,H, d}}{n}\right)^2 n^{1-Hd}    \int^{nb_{\ell+1}}_{n a_{\ell+1}} (s^{\ell+1}_1-nb_{\ell})^{-Hd-2H\beta} \, ds^{\ell+1}_1.         
\end{align*}
Recall that $b_{\ell}<a_{\ell+1}$ and the definition of $\ell_{n,H, d}$ in (\ref{switch}). Then
\begin{align}  \label{iin}
\text{II}_n &\leq c_5
\begin{cases}
n^{1-Hd-2H\beta} &  \text{if}\quad   1-Hd<2H\beta;  \\  \\
\ln^{-1} n&  \text{if}\quad   1-Hd=2H\beta. 
\end{cases}
\end{align}
Therefore, $\lim\limits_{n\to\infty}\E (G_n)=0$. This completes the proof.
\end{proof}

\bigskip
 Consider now  the convergence of moments when all exponents $m_i$ are even.  
 \begin{proposition} \label{even1} Suppose that $\frac{1}{2\beta+d}<H<\frac{1}{d}$ for some $\beta\in (0,2]$ and all exponents $m_i$ are even. Then
 \begin{equation}  \label{c1}
\lim_{n\to\infty}\E(G_n)=C_{H,d,f}^{\frac{|\mathbf{m}|}{2}}\, \mathbb{E} \Big(
\prod_{i=1}^N   \big( W(L_{b_i}(\lambda))- W(L_{a_i}(\lambda)) \big)^{m_i}
\Big).
\end{equation}
%where the expectation in the right-hand side of the above equation is given in Lemma .
\end{proposition}
\begin{proof} Recall the definition of $D_\mathbf{m}$ in \eref{e.3.1.4}. Let 
\begin{equation}
D_{\mathbf{m},n}=D_\mathbf{m}\bigcap \Big(\bigcup^N_{i=1}\bigcup^{m_i/2}_{\ell=1}\big\{s \in \R^{|\mathbf{m}|}: s^i_{2\ell} -s^i_{2\ell-1}>n^{\frac{1}{4}}\,\text{or}\, s^i_{2\ell-1} -s^i_{2\ell-2}<n^{\frac{1}{2}}\big\}\Big),
\end{equation}
where $s^i_0=s^{i-1}_{m_{i-1}}$ for $i=2,\dots, N$ and $s^1_0=0$.

Note that \begin{align*}
\lim_{n\to\infty}\E (G_n)
&=\frac{\mathbf{m}!}{(2\pi)^{|\mathbf{m}|d}}\,  \lim_{n\to\infty} n^{\frac{|\mathbf{m}|(Hd-1)}{2}}  \int_{\R^{|\mathbf{m}|d}}  \int_{D_\mathbf{m}} 
           \Big( \prod_{(i,j)\in J_e} |\widehat{f}(x^i_j)-\widehat{f}(0)|^2\Big)\\
&\qquad \times \exp\bigg(-\iota n^H\lambda\cdot x^1_1-\frac{1}{2}\Var\Big( \sum^{N}_{i=1}\sum^{m_i}_{j=1} x^i_j\cdot \big(X_{s^i_j}-X_{s^i_{j-1}}\big) \Big)\bigg) \, ds\, dx\\
&:=\lim_{n\to\infty} I_{\mathbf{m}, n}.
\end{align*}

\medskip
\noindent
{\bf Step 1} We show that
\begin{align} \label{K1}
\limsup_{n\to\infty}\big| I_{\mathbf{m}, n,1}\big|=0,
\end{align}
where
\begin{align*}
I_{\mathbf{m}, n, 1}
&=\frac{\mathbf{m}!}{(2\pi)^{|\mathbf{m}|d}}\, n^{\frac{|\mathbf{m}|(Hd-1)}{2}}  \int_{\R^{|\mathbf{m}|d}}  \int_{D_{\mathbf{m}, n}} 
           \Big( \prod_{(i,j)\in J_e} |\widehat{f}(x^i_j)-\widehat{f}(0)|^2\Big)\\
&\qquad\qquad\times \exp\bigg(-\iota n^H\lambda\cdot x^1_1-\frac{1}{2}\Var\Big( \sum^{N}_{i=1}\sum^{m_i}_{j=1} x^i_j\cdot \big(X_{s^i_j}-X_{s^i_{j-1}}\big) \Big)\bigg) \, ds\, dx.
\end{align*}
  
In order to prove (\ref{K1}), set  $|\mathbf{m}|=2m$.  For $\ell=1,\dots, m$,  we introduce the set 
\[
D_{2m,n}^\ell= \Big\{(s_1, \dots, s_{|\mathbf{m}|}) \in D_{\mathbf{m}}:\, s_1<s_2<\cdots<s_{|\mathbf{m}|},\, s_{2\ell} -s_{2\ell-1} > n^{\frac{1}{4}}\; \text{or}\; s_{2\ell-1}-s_{2\ell-2}<n^{\frac{1}{2}}\Big\},
\] 
where $s_0=0$.

Then, it suffices to show that
\begin{align} \label{dl2mn}
&\lim_{n\to\infty}\, n^{m(Hd-1)}\int_{\R^{2md}}\int_{D^{\ell}_{2m,n}} \Big(\prod^{m}_{i=1}| \widehat{f}(x_{2i})-\widehat{f}(0)|^2\Big) \nonumber \\
&\qquad\qquad\times \exp\Big(-\frac{1}{2}\Var\big(\sum\limits^{2m}_{i=1} x_i\cdot (X_{s_i}-X_{s_{i-1}})\big)\Big)\, ds\, dx=0
\end{align}
for each $\ell=1,\dots,m$.

The left hand side of (\ref{dl2mn}) is positive and less  than or equal to
\[
\lim_{n\to\infty}\, n^{m(Hd-1)}\int_{\R^{2md}}\int_{D^{\ell}_{2m,n}} \Big(\prod^{m}_{i=1} |\widehat{f}(x_{2i})-\widehat{f}(0)|^2\Big)\, \exp\Big(-\frac{\kappa_{H,2m}}{2}\sum\limits^{2m}_{i=1} |x_i|^2(s_i-s_{i-1})^{2H}\Big)\, ds\, dz=0,
\]
where we use the definition of $D^{\ell}_{2m,n}$, $1-Hd-2H\beta<0$ and Lemma \ref{f} in the last equality.

\medskip
\noindent
{\bf Step 2} On the set $D_{\mathbf{m}}-D_{\mathbf{m},n}$, by Assumptions ({\bf A})\&({\bf B}), we could obtain that, for $n$ large enough,
\begin{align*}
&\Var\Big( \sum^{N}_{i=1}\sum^{m_i}_{j=1} x^i_j\cdot \big(X_{s^i_j}-X_{s^i_{j-1}}\big) \Big)\\
&=(\sigma+o(1))\sum^{N}_{i=1}\sum^{m_i/2}_{\ell=1} |x^i_{2\ell}|^2 (s^i_{2\ell}-s^i_{2\ell-1})^{2H}+(1+o(1))\Var\Big( \sum^{N}_{i=1}\sum^{m_i/2}_{\ell=1} x^i_{2\ell-1}\cdot \big(X_{s^i_{2\ell-1}}-X_{s^i_{2\ell-2}}\big) \Big)\\
&=(\sigma+o(1))\sum^{N}_{i=1}\sum^{m_i/2}_{\ell=1} |x^i_{2\ell}|^2 (s^i_{2\ell}-s^i_{2\ell-1})^{2H}+(1+o(1))\Var\Big( \sum^{N}_{i=1}\sum^{m_i/2}_{\ell=1} x^i_{2\ell-1}\cdot \big(X_{s^i_{2\ell-1}}-X_{s^i_{2\ell-3}}\big) \Big),
\end{align*}
where $s^i_{-1}=s^{i-1}_{m_{i-1}-1}$ for $i=2,\dots, N$ and $s^1_{-1}=0$.

Therefore, by the dominated convergence theorem, the self-similarity of the Gaussian process $X$ and Lemma \ref{lema2}, we could obtain that 
\begin{align*}
&\lim_{n\to\infty} (I_{\mathbf{m}, n}-I_{\mathbf{m}, n,1})\\
&=\ \left(\frac{2}{(2\pi)^d} \int_{\R^d} \int^{\infty}_0 |\widehat{f}(y)-\widehat{f}(0)|^2 e^{-\frac{\sigma}{2}|y|^2u^{2H}}\, du\, dy \right)^{\frac{|\mathbf{m}|}{2}}\frac{\mathbf{m}!}{2^{\frac{|\mathbf{m}|}{2}}(2\pi)^{\frac{|\mathbf{m}|d}{2}}}\\
&\qquad\times  \int_{\R^{\frac{|\mathbf{m}|d}{2}}}\int_{\prod\limits^{N}_{i=1} [a_i,b_i]^{\frac{m_i}{2}}_<} \exp\bigg(-\iota \lambda\cdot x^1_1-\frac{1}{2}\Var\Big( \sum^{N}_{i=1}\sum^{m_i/2}_{\ell=1} x^i_{2\ell-1}\cdot \big(X_{s^i_{2\ell-1}}-X_{s^i_{2\ell-3}}\big) \Big)\bigg) \, ds\, dx\\
&=C_{H,d,f}^{\frac{|\mathbf{m}|}{2}}\, \mathbb{E} \Big(\prod_{i=1}^N   \big( W(L_{b_i}(\lambda))- W(L_{a_i}(\lambda)) \big)^{m_i}\Big),
\end{align*}
where $[a_i,b_i]^{\frac{m_i}{2}}_<=\{a_i<s^i_1<s^i_3\dots<s^i_{m_i-1}<b_i\}$ for $i=1,\dots, N$.

\medskip
\noindent
{\bf Step 3} Combing the results in {\bf Step 1} and {\bf Step 2} gives the desired convergence of moments when all exponents $m_i$ are even.  
\end{proof}

\begin{proposition} \label{even2} Suppose that $\frac{1}{2\beta+d}=H$ for some $\beta\in \{1,2\}$ and all exponents $m_i$ are even. Then
\begin{equation}  \label{c1}
\lim_{n\to\infty}\E(G_n)=D_{H,d,f}^{\frac{|\mathbf{m}|}{2}}\, \mathbb{E} \Big(
\prod_{i=1}^N   \big( W(L_{b_i}(\lambda))- W(L_{a_i}(\lambda)) \big)^{m_i}
\Big).
\end{equation}
%where the expectation in the right-hand side of the above equation is given in Lemma .
\end{proposition}
\begin{proof} Recall the definition of $D_\mathbf{m}$ in \eref{e.3.1.4}. Let 
\begin{equation}
\widetilde{D}_{\mathbf{m}, n}=D_\mathbf{m}\bigcap \Big(\bigcup^N_{i=1}\bigcup^{m_i/2}_{\ell=1}\big\{s \in \R^{|\mathbf{m}|}: s^i_{2\ell} -s^i_{2\ell-1}>\frac{n}{\ln^2 n}\,\text{or}\, s^i_{2\ell-1} -s^i_{2\ell-2}<\frac{n}{\ln n} \big\}\Big),
\end{equation}
where $s^i_0=s^{i-1}_{m_{i-1}}$ for $i=2,\dots, N$ and $s^1_0=0$.

Note that \begin{align*}
\lim_{n\to\infty}\E (G_n)
&=\frac{\mathbf{m}!}{(2\pi)^{|\mathbf{m}|d}}\,  \lim_{n\to\infty} (\ln n)^{-\frac{|\mathbf{m}|}{2}} n^{\frac{|\mathbf{m}|(Hd-1)}{2}}  \int_{\R^{|\mathbf{m}|d}}  \int_{D_\mathbf{m}} 
           \Big( \prod_{(i,j)\in J_e} |\widehat{f}(x^i_j)-\widehat{f}(0)|^2\Big)\\
&\qquad\qquad\times \exp\bigg(-\iota n^H\lambda\cdot x^1_1-\frac{1}{2}\Var\Big( \sum^{N}_{i=1}\sum^{m_i}_{j=1} x^i_j\cdot \big(X_{s^i_j}-X_{s^i_{j-1}}\big) \Big)\bigg) \, ds\, dx\\
&:=\lim_{n\to\infty} I_{\mathbf{m}, n}.
\end{align*}

\medskip
\noindent
{\bf Step 1} We show that
\begin{align} \label{K2}
\lim_{n\to\infty}\big| I_{\mathbf{m}, n, 2}\big|=0,
\end{align}
where
\begin{align*}
I_{\mathbf{m}, n, 2}
&=\frac{\mathbf{m}!}{(2\pi)^{|\mathbf{m}|d}}\, (\ln n)^{-\frac{|\mathbf{m}|}{2}} n^{\frac{|\mathbf{m}|(Hd-1)}{2}}  \int_{\R^{|\mathbf{m}|d}}  \int_{\widetilde{D}_{\mathbf{m}, n}} 
           \Big( \prod_{(i,j)\in J_e} |\widehat{f}(x^i_j)-\widehat{f}(0)|^2\Big)\\
&\qquad\qquad\times \exp\bigg(-\iota n^H\lambda\cdot x^1_1-\frac{1}{2}\Var\Big( \sum^{N}_{i=1}\sum^{m_i}_{j=1} x^i_j\cdot \big(X_{s^i_j}-X_{s^i_{j-1}}\big) \Big)\bigg) \, ds\, dx.
\end{align*}
  
In order to prove (\ref{K2}), set  $|\mathbf{m}|=2m$.  For $\ell=1,\dots, m$,  we introduce the set 
\[
\widetilde{D}_{2m,n}^\ell= \Big\{(s_1, \dots, s_{|\mathbf{m}|}) \in D_{\mathbf{m}}:\, s_1<s_2<\cdots<s_{|\mathbf{m}|},\, s_{2\ell} -s_{2\ell-1} > \frac{n}{\ln^2 n}\; \text{or}\; s_{2\ell-1}-s_{2\ell-2}<\frac{n}{\ln n}\Big\},
\] 
where $s_0=0$.

Then, it suffices to show that
\begin{align} \label{D2mn}
&\lim_{n\to\infty}\, (\ln n)^{-m}\, n^{m(Hd-1)}\int_{\R^{2md}}\int_{\widetilde{D}^{\ell}_{2m,n}} \Big(\prod^{m}_{i=1}| \widehat{f}(x_{2i})-\widehat{f}(0)|^2\Big) \nonumber \\
&\qquad\qquad\qquad\qquad\qquad \times \exp\Big(-\frac{1}{2}\Var\big(\sum\limits^{2m}_{i=1} x_i\cdot (X_{s_i}-X_{s_{i-1}})\big)\Big)\, ds\, dx=0
\end{align}
for each $\ell=1,\dots,m$.

The left hand side of (\ref{D2mn}) is positive and less  than or equal to
\begin{align*}
&\lim_{n\to\infty}\, (\ln n)^{-m}\,  n^{m(Hd-1)}\int_{\R^{2md}}\int_{\widetilde{D}^{\ell}_{2m,n}} \Big(\prod^{m}_{j=1} |\widehat{f}(z_{2j})-\widehat{f}(0)|^2\Big)\\
&\qquad\qquad\qquad\qquad \times \exp\Big(-\frac{\kappa_{H,2m}}{2}\sum\limits^{2m}_{i=1} |z_i|^2(s_i-s_{i-1})^{2H}\Big)\, ds\, dz=0,
\end{align*}
where we use the definition of $\widetilde{D}^{\ell}_{2m,n}$, $1-Hd=2\beta$, and Lemma \ref{f} in the last equality.

\medskip
\noindent
{\bf Step 2} On the set $D_{\mathbf{m}}-\widetilde{D}_{\mathbf{m}, n}$, by Assumptions ({\bf A})\&({\bf B}), we could obtain that, for $n$ large enough,
\begin{align*}
&\Var\Big( \sum^{N}_{i=1}\sum^{m_i}_{j=1} x^i_j\cdot \big(X_{s^i_j}-X_{s^i_{j-1}}\big) \Big)\\
&=(\sigma+o(1))\sum^{N}_{i=1}\sum^{m_i/2}_{\ell=1} |x^i_{2\ell}|^2 (s^i_{2\ell}-s^i_{2\ell-1})^{2H}+(1+o(1))\Var\Big( \sum^{N}_{i=1}\sum^{m_i/2}_{\ell=1} x^i_{2\ell-1}\cdot \big(X_{s^i_{2\ell-1}}-X_{s^i_{2\ell-2}}\big) \Big)\\
&=(\sigma+o(1))\sum^{N}_{i=1}\sum^{m_i/2}_{\ell=1} |x^i_{2\ell}|^2 (s^i_{2\ell}-s^i_{2\ell-1})^{2H}+(1+o(1))\Var\Big( \sum^{N}_{i=1}\sum^{m_i/2}_{\ell=1} x^i_{2\ell-1}\cdot \big(X_{s^i_{2\ell-1}}-X_{s^i_{2\ell-3}}\big) \Big).
\end{align*}

Therefore,  by the dominated convergence theorem, the self-similarity of the Gaussian process $X$ and Lemma \ref{lema2}, we could obtain that 
\begin{align*}
&\lim_{n\to\infty} (I_{\mathbf{m}, n}-I_{\mathbf{m}, n, 2})\\
&=\left(\frac{2}{(2\pi)^d} \lim_{n\to \infty} \frac{1}{\ln n}\int^n_0 \int_{\R^d} |\widehat{f}(y)-\widehat{f}(0)|^2 e^{-\frac{\sigma}{2}|y|^2 u^{2H}}\, dy\, du \right)^{\frac{|\mathbf{m}|}{2}}\frac{\mathbf{m}!}{2^{\frac{|\mathbf{m}|}{2}}(2\pi)^{\frac{|\mathbf{m}|d}{2}}}\\
&\qquad\times  \int_{\R^{\frac{|\mathbf{m}|d}{2}}}\int_{\prod\limits^{N}_{i=1} [a_i,b_i]^{\frac{m_i}{2}}_<} \exp\bigg(-\iota \lambda\cdot x^1_1-\frac{1}{2}\Var\Big( \sum^{N}_{i=1}\sum^{m_i/2}_{\ell=1} x^i_{2\ell-1}\cdot \big(X_{s^i_{2\ell-1}}-X_{s^i_{2\ell-3}}\big) \Big)\bigg) \, ds\, dx\\
&=D_{H,d,f}^{\frac{|\mathbf{m}|}{2}}\, \mathbb{E} \Big(\prod_{i=1}^N   \big( W(L_{b_i}(\lambda))- W(L_{a_i}(\lambda)) \big)^{m_i}\Big),
\end{align*}
where $[a_i,b_i]^{\frac{m_i}{2}}_<=\{a_i<s^i_1<s^i_3\dots<s^i_{m_i-1}<b_i\}$ for $i=1,\dots, N$.

\medskip
\noindent
{\bf Step 3} Combing the results in {\bf Step 1} and {\bf Step 2} gives the desired convergence of moments when all exponents $m_i$ are even.  
\end{proof}

\medskip 
 
{\medskip \noindent \textbf{Proof of Theorem \ref{th1}.}}    This follows
from Lemma \ref{lema2}, Propositions \ref{tight}, \ref{odd} and \ref{even1} by the method of moments.

\medskip

{\medskip \noindent \textbf{Proof of Theorem \ref{th2}.}}    This follows
from Lemma \ref{lema2}, Propositions \ref{tight}, \ref{odd} and \ref{even2} by the method of moments.
 
\medskip

{\medskip \noindent \textbf{Proof of Theorem \ref{th3}.}}  By Remark 1.4 in \cite{hx} and Lemma \ref{lndp}, it is easy to see that $L^{({\beta \bf e}_i)}_t(\lambda)$ exists in $L^p$ for any $p\geq 1$. For $\beta=1,2$, we define 
\[
A^{\beta}_{n,t}=n^{H\beta} \Big(n^H\int^t_0 f(n^H(X_s-\lambda))\, ds-L_t(\lambda)\int_{\mathbb{R}^d} f(x)\, dx\Big)-\frac{1}{\beta}\int_{\mathbb{R}^d} \Big(\sum^d_{i=1} x^{\beta}_i L^{({\beta \bf e}_i)}_t(\lambda) \Big) f(x)\, dx.
\]

\medskip
\noindent
{\bf Step 1} We first consider the case $\beta=1$. For any $m\in \mathbb{N}$, we have
\begin{align*}
\mathbb{E}|A^{1}_{n,t}|^{2m}
&=\frac{1}{(2\pi)^{2md}}\int_{[0,t]^{2m}}\int_{\mathbb{R}^{2md}}\prod^{2m}_{j=1}\Big[n^{H} \big(\widehat{f}(\frac{u^j}{n^H})-\widehat{f}(0)\big)-\iota \sum^d_{i=1}\partial_i\widehat{f}(0) u^j_i \Big] \\
&\qquad\qquad\times \exp\left(-\frac{1}{2}\Var\Big(\sum^{2m}_{j=1} u^j\cdot(X_{s_j}-\lambda)\Big) \right)\, du\, ds,
\end{align*}
where $\partial_i\widehat{f}(x)=\frac{\partial }{\partial x_i}\widehat{f}(x)$ for any $x=(x_1,\dots, x_d)\in\mathbb{R}^d$.

Note that 
\begin{align*}
\Big|n^{H} \big(\widehat{f}(\frac{u^j}{n^H})-\widehat{f}(0)\big)-\iota \sum^d_{i=1}\partial_i\widehat{f}(0) u^j_i\Big|\leq c_1|u^j|
\end{align*}
for all  $u^j=(u^j_1,\dots, u^j_d)\in\mathbb{R}^d$.  

Moreover, using the local nondeterminism property (\ref{lndp}) and making the change of variables $v^j=\sum\limits^{2m}_{\ell=j} u^{\ell}$, we could easily obtain that 
\begin{align*}
&\int_{[0,t]^{2m}}\int_{\mathbb{R}^{2md}}\prod^{2m}_{j=1}|u^j| \exp\left(-\frac{1}{2}\Var\Big(\sum^{2m}_{j=1} u^j\cdot(X_{s_j}-\lambda)\Big) \right)\, du\, ds\\
&\leq c_2 \int_{[0,t]^{2m}_{<}}\int_{\mathbb{R}^{2md}}\prod^{2m}_{j=1}|v^j-v^{j+1}| \exp\left(-\frac{\kappa_{H,2m}}{2}\sum^{2m}_{j=1} |v^j|^2(s_j-s_{j-1})^{2H}\right)\, dv\, ds\\
&<\infty,
\end{align*}
where we use $1-2H-Hd>0$ in the last inequality. 

Therefore, by the dominated convergence theorem, $\lim\limits_{n\to\infty}\mathbb{E}|A^{1}_{n,t}|^{2m}=0$.

\medskip
\noindent
{\bf Step 2} Now we consider the case $\beta=2$. For any $m\in \mathbb{N}$, we have
\begin{align*}
\mathbb{E}|A^{2}_{n,t}|^{2m}
&=\frac{1}{(2\pi)^{2md}}\int_{[0,t]^{2m}}\int_{\mathbb{R}^{2md}}\prod^{2m}_{j=1}\Big[n^{2H} \big(\widehat{f}(\frac{u^j}{n^H})-\widehat{f}(0)\big)+\frac{1}{2}\sum^d_{i,k=1}\partial^2_{i,k}\widehat{f}(0) u^j_i u^j_k \Big] \\
&\qquad\qquad\times \exp\left(-\frac{1}{2}\Var\Big(\sum^{2m}_{j=1} u^j\cdot(X_{s_j}-\lambda)\Big) \right)\, du\, ds,
\end{align*}
where $\partial^2_{i,k}\widehat{f}(x)=\frac{\partial^2 }{\partial x_i \partial x_k}\widehat{f}(x)$ for any $x=(x_1,\dots, x_d)\in\mathbb{R}^d$.

Recall that $\int_{\R^d} xf(x)dx=0$. Hence
\begin{align*}
&\mathbb{E}|A^{2}_{n,t}|^{2m}\\
&=\frac{1}{(2\pi)^{2md}}\int_{[0,t]^{2m}}\int_{\mathbb{R}^{2md}}\prod^{2m}_{j=1}\Big[n^{2H} \big(\widehat{f}(\frac{u^j}{n^H})-\widehat{f}(0)\big)-\iota \sum^d_{i=1}\partial_i\widehat{f}(0) \frac{u^j_i}{n^H} +\frac{1}{2}\sum^d_{i,k=1}\partial^2_{i,k}\widehat{f}(0) u^j_i u^j_k \Big] \\
&\qquad\qquad\times \exp\left(-\frac{1}{2}\Var\Big(\sum^{2m}_{j=1} u^j\cdot(X_{s_j}-\lambda)\Big) \right)\, du\, ds.
\end{align*}
Note that 
\begin{align*}
\left|n^{2H} \big(\widehat{f}(\frac{u^j}{n^H})-\widehat{f}(0)\big)-\iota \sum^d_{i=1}\partial_i\widehat{f}(0) \frac{u^j_i}{n^H} +\frac{1}{2}\sum^d_{i,k=1}\partial^2_{i,k}\widehat{f}(0) u^j_i u^j_k\right|\leq c_3|u^j|^2
\end{align*}
for all $u^j\in\mathbb{R}^d$.

Moreover, using the local nondeterminism property (\ref{lndp}) and making the change of variables $v^j=\sum\limits^{2m}_{\ell=j} u^{\ell}$, we could easily obtain that 
\begin{align*}
&\int_{[0,t]^{2m}}\int_{\mathbb{R}^{2md}}\prod^{2m}_{j=1}|u^j|^2\exp\left(-\frac{1}{2}\Var\Big(\sum^{2m}_{j=1} u^j\cdot(X_{s_j}-\lambda)\Big) \right)\, du\, ds\\
&\leq c_4 \int_{[0,t]^{2m}_{<}}\int_{\mathbb{R}^{2md}}\prod^{2m}_{j=1}|v^j-v^{j+1}|^2 \exp\left(-\frac{\kappa_{H,2m}}{2}\sum^{2m}_{j=1} |v^j|^2(s_j-s_{j-1})^{2H}\right)\, dv\, ds\\
&<\infty,
\end{align*}
where we use $1-4H-Hd>0$ in the last inequality. 

Therefore, by the dominated convergence theorem, $\lim\limits_{n\to\infty}\mathbb{E}|A^{2}_{n,t}|^{2m}=0$. 

\medskip
\noindent
{\bf Step 3}
 Combining the results in {\bf Step 1} and {\bf Step 2} gives the desired $L^p$ convergence. This completes the proof.

\section{Appendix A}

Recall that the components of $X=\{X_t=(X^1_t,\dots, X^d_t),\, t\geq 0\}$ are independent copies a one-dimensional centered Gaussian process. To give sufficient conditions for $X$ to satisfy the strong local nondeterminism property (\ref{slndp}), we only need to consider under what conditions the underlying one-dimensional Gaussian process satisfies the strong local nondeterminism property (\ref{slndp}). Assume that $Z=(Z_t,\, t\geq 0)$ is a one-dimensional centered Gaussian process satisfying the self-similarity with index $H\in (0,1)$.  Let $Y=(Y_t,\, t\in\mathbb{R})$ where $Y_t=e^{-Ht}Z_{e^t}$ for each $t\in\mathbb{R}$.  Then it is easy to see that $Y$ is a one-dimensional centered Gaussian stationary process with covariance function $r(t)=\mathbb{E}[Y_0Y_t]=e^{-Ht}\mathbb{E}[Z_1Z_{e^t}]$ for any $t\in\mathbb{R}$.  Clearly, $r(\cdot)$ is an even function on $\mathbb{R}$. In the sequel, we will give a sufficient condition for the Gaussian process $Z$ to possess the strong local nondeterminism property (\ref{slndp}). 

\begin{theorem} \label{glndp} Suppose that $r(t)$ is a non-negative decreasing function on $[0,\infty)$, $r(\cdot)\in L^1(\mathbb{R})$ and there exists a positive constant $c$ such that 
\[
\int^{\infty}_0 r(t)\cos(t\lambda) dt\geq \frac{c}{(|\lambda|+1)^{1+2H}}
\]
for all $\lambda\in\mathbb{R}$. Then there exists a positive constant $\kappa_{H}$ such that for any integer $m\ge 1$, any times $0=s_0<s_1\leq \dots \leq
s_m<t<\infty$,  
\begin{align} \label{slnpb}
\Var\Big(Z_t| Z_{s_1},\dots, Z_{s_m}\Big)\geq \kappa_{H} \min_{1\leq j\leq m} |t-s_j|^{2H}. 
\end{align} 
\end{theorem}

\begin{proof}
By assumption, the spectral density $f$ of $Y$ has the expression 
\begin{align*}  
f(\lambda)=\frac{1}{\pi}\int^{\infty}_0 r(t)\cos(t\lambda) dt.
\end{align*}

For any $0<r<t$, by the self-similarity of $Z$,
\begin{align*}
\Var\Big(Z_t| Z_{s}, 0<s<t-r\Big)
&=t^{2H}\Var\Big(Z_1| Z_{s}, 0<s<1-r'\Big)\\
&=t^{2H}\Var\Big(Y_0| Y_{\ln s}, 0<s<1-r'\Big)\\
&\geq t^{2H}\Var\Big(Y_0| Y_u, |u|\geq \frac{r'}{2}\Big),
\end{align*}
where $r'=r/t$ and in the last inequality we use $\ln s<\ln (1-r')<-\frac{r'}{2}$ for $r'\in (0,1)$.

Using similar arguments as in the proof of Lemma 1 in \cite{cd} gives the desired result.
\end{proof}

Next, we will show that bi-fractional Brownian motion $B^{H_0, K_0}$ and sub-fractional Brownian motion $S^{H}$ satisfy the strong local nondeterminism property (\ref{slndp}).

\begin{corollary} \label{bfbm}
Assume that $B^{H_0, K_0}$ is the one-dimensional bi-fractional Brownian motion. Then there exists a positive constant $\kappa_{H_0, K_0}$ such that for any integer $m\ge 1$, any times $0=s_0<s_1\leq \dots \leq
s_m<t<\infty$,  
\begin{align*}  
\Var\Big(B^{H_0, K_0}_t| B^{H_0, K_0}_{s_1},\dots, B^{H_0, K_0}_{s_m}\Big)\geq \kappa_{H_0, K_0} \min_{1\leq j\leq m} |t-s_j|^{2H_0K_0}. 
\end{align*} 
\end{corollary}

\begin{proof}
Note that $B^{H_0, K_0}$ is a one-dimensional centered Gaussian process satisfying the self-similarity with the index $H_0K_0\in (0,1)$ and 
\begin{align*}
r(t):=e^{-H_0K_0t}\mathbb{E}[B^{H_0, K_0}_1 B^{H_0, K_0}_{e^{t}}]=\frac{e^{H_0K_0 t}}{2^{K_0}}[(1+e^{-2H_0 t})^{K_0}-|1-e^{-t}|^{2H_0K_0}].
\end{align*}
By the proof of Proposition 2.1 in \cite{tx}, $r(\cdot )\in L^1(\mathbb{R})$ and 
\begin{align} \label{boundb}
\lim\limits_{\lambda\to \infty} \frac{\int^{\infty}_0 r(t)\cos(t\lambda) dt}{|\lambda|^{-1-2H_0K_0}}>0.
\end{align}

In the sequel, we will show that $r(t)$ is a positive and strictly decreasing function on $[0,\infty)$. For any $t\geq 0$, it is easy to see that 
\begin{align*}
r(t)>\frac{e^{H_0K_0 t}}{2^{K_0}}[(1+e^{-2H_0 t})^{K_0}-1]>0.
\end{align*}
Moreover, for any $t>0$, 
\begin{align*}
r'(t)
&=\frac{H_0K_0}{2^{K_0}}e^{H_0K_0 t}\Big[(1+e^{-2H_0 t})^{K_0-1}(1-e^{-2H_0t})-(1+e^{-t})(1-e^{-t})^{2H_0K_0-1}\Big]\\
&<\frac{H_0K_0}{2^{K_0}}e^{H_0K_0 t}\Big[(1-(e^{-t})^{2})-(1+e^{-t})(1-e^{-t})^{2H_0K_0-1}\Big]\\
&<\frac{H_0K_0}{2^{K_0}}e^{H_0K_0 t}(1+e^{-t}) \Big[(1-e^{-t})-(1-e^{-t})^{2H_0K_0-1}\Big]\\
&<0,
\end{align*}
where we use $K_0\in (0,1]$ and $H_0\in (0,1)$ in the first inequality,  and $2H_0K_0-1<1$ in the last inequality. 

Note that $\int^{\infty}_0 r(t)\cos(t\lambda) dt$ is even and continuous. According to the graph of the cosine function on $[0,\infty)$ and the fact that $r(t)$ is a positive and strictly decreasing function on $[0,\infty)$, $\int^{\infty}_0 r(t)\cos(t\lambda) dt>0$ for all $\lambda\in\mathbb{R}$. Then, (\ref{boundb}) implies that there exists a positive constant $c$ such that
\begin{align*}
\int^{\infty}_0 r(t)\cos(t\lambda) dt \geq \frac{c}{(|\lambda|+1)^{1+2H_0K_0}}
\end{align*}
for all $\lambda\in\mathbb{R}$. Using Theorem \ref{glndp} gives the desired result.
\end{proof}

\begin{corollary} \label{sfbm}
Assume that $S^H$ is the one-dimensional sub-fractional Brownian motion. Then there exists a positive constant $\kappa_{H}$ such that for any integer $m\ge 1$, any times $0=s_0<s_1\leq \dots \leq
s_m<t<\infty$,  
\begin{align*}  
\Var\Big(S^H_t| S^H_{s_1},\dots, S^H_{s_m}\Big)\geq \kappa_{H}  \min_{1\leq j\leq m} |t-s_j|^{2H}. 
\end{align*} 
\end{corollary}

\begin{proof}
Note that $S^H$ is a one-dimensional centered Gaussian process satisfying the self-similarity with the index $H\in (0,1)$ and 
\begin{align*}
r(t)=e^{-Ht}\mathbb{E} [S^H_1S^H_{e^{t}}]=e^{Ht}\Big[e^{-2Ht}+1-\frac{1}{2}[(1+e^{-t})^{2H}+|1-e^{-t}|^{2H}\Big].
\end{align*}
By the proof of Theorem 1 in \cite{l}, $r(\cdot )\in L^1(\mathbb{R})$ and
\begin{align} \label{bounds}
\lim\limits_{\lambda\to \infty} \frac{\int^{\infty}_0 r(t)\cos(t\lambda) dt}{|\lambda|^{-1-2H}}>0.
\end{align}

In the sequel, we will show that $r(t)$ is a positive and strictly decreasing function on $[0,\infty)$. For any $t\geq 0$, if $2H\leq 1$, then it is easy to see that 
\begin{align*}
r(t)
&=e^{Ht}\Big[e^{-2Ht}+1-\frac{1}{2}[(1+e^{-t})^{2H}+(1-e^{-t})^{2H}]\Big]\\
&\geq e^{Ht}\Big[e^{-2Ht}+1-\frac{1}{2}[1+e^{-2Ht}+1]\Big]\\
&=\frac{e^{-Ht}}{2}>0.
\end{align*}
%where in the first inequality we use $(1+e^{-t})^{2H}\leq 1+e^{-2Ht}$ for any $t\geq 0$ because $2H\in (0,1)$.

If $2H>1$, let 
\[
r_1(t)=e^{-2Ht}+1-\frac{1}{2}[(1+e^{-t})^{2H}+(1-e^{-t})^{2H}].
\]
Clearly, $r_1(t)$ is continuous on $[0,\infty)$ and $\lim\limits_{t\to\infty}r_1(t)=0$. For any $t>0$,
\begin{align*}
r'_1(t)
%&=-2He^{-2Ht}-\frac{1}{2}[-2H(1+e^{-t})^{2H-1}e^{-t}+2H(1-e^{-t})^{2H-1}e^{-t}]\\
&=He^{-t}\Big[(1+e^{-t})^{2H-1}-(1-e^{-t})^{2H-1}-2(e^{-t})^{2H-1}\Big]\\
&\leq He^{-t}\Big[1+(e^{-t})^{2H-1}-(1-e^{-t})^{2H-1}-2(e^{-t})^{2H-1}\Big]\\
&=He^{-t}\Big[1-(1-e^{-t})^{2H-1}-(e^{-t})^{2H-1}\Big]\\
&<0,
\end{align*}
where in the two inequalities we use $2H-1\in (0,1)$ and 
\[
(a+b)^{2H-1}<a^{2H-1}+b^{2H-1}
\]
for all $a,b>0$. Therefore, in the case $2H>1$, $r(t)>0$ for all $t\geq 0$.

Moreover, for any $t>0$, 
\begin{align*}
r'(t)
&=He^{Ht}\Big[1-e^{-2Ht}-\frac{1}{2}(1+e^{-t})^{2H-1}(1-e^{-t})-\frac{1}{2}(1-e^{-t})^{2H-1}(1+e^{-t})\Big]\\
&\leq He^{Ht}\Big[(1-e^{-2t})^H-\frac{1}{2}(1+e^{-t})^{2H-1}(1-e^{-t})-\frac{1}{2}(1-e^{-t})^{2H-1}(1+e^{-t})\Big]\\
%&=He^{Ht}(1-e^{-2t})^H\left[1-\frac{1}{2}\left[\left(\frac{1-e^{-t}}{1+e^{-t}}\right)^{1-H}+\left(\frac{1+e^{-t}}{1-e^{-t}}\right)^{1-H}\right]\right]\\
%&=\frac{H}{2}e^{Ht}(1-e^{-2t})^{2H-1}\left[2(1-e^{-2t})^{1-H}-(1-e^{-t})^{2(1-H)}-(1+e^{-t})^{2(1-H)} \right]\\
&=-\frac{H}{2}e^{Ht}(1-e^{-2t})^{2H-1}\Big[(1+e^{-t})^{1-H}-(1-e^{-t})^{1-H}\Big]^2 \\
&<0.
\end{align*}

Note that $\int^{\infty}_0 r(t)\cos(t\lambda) dt$ is even and continuous.  According to the graph of the cosine function on $[0,\infty)$ and the fact that $r(t)$ is a positive and strictly decreasing function on $[0,\infty)$, $\int^{\infty}_0 r(t)\cos(t\lambda) dt>0$ for all $\lambda\in\mathbb{R}$. Then, (\ref{bounds}) implies that there exists a positive constant $c$ such that
\begin{align*}
\int^{\infty}_0 r(t)\cos(t\lambda) dt \geq \frac{c}{(|\lambda|+1)^{1+2H}}
\end{align*}
for all $\lambda\in\mathbb{R}$. Using Theorem \ref{glndp} gives the desired result.
\end{proof}

\section{Appendix B}

Given $N\in\mathbb{N}$. The following result gives a sufficient condition for a distribution function on $\mathbb{R}^N$ to be determined by its moments. It is an extension of Theorem 3.3.11 in \cite{D} which corresponds to the case $N=1$ here.

\begin{theorem} \label{moment}
If $\limsup\limits_{k\to\infty} \frac{\mu^{\frac{1}{2k}}_{i,2k}}{2k}=r_i<\infty$ for each $i=1,2,\dots,N$, then there is at most one distribution function $F$ on $\mathbb{R}^N$ with $\mu_{i,k}=\int_{\mathbb{R}^N} x^k_i dF(x)$ for all positive integers $k$ and $i=1,2,\dots, N$.
\end{theorem}

\begin{proof}
Let $F$ be any distribution function on $\mathbb{R}^N$ with moments $\mu_{i,k}$ and let $\nu_{i,k}=\int_{\mathbb{R}^N} |x_i|^k dF(x)$. It is well known that $\nu^{\frac{1}{k}}_{i,k}\leq \nu^{\frac{1}{\ell}}_{i,\ell}$ if $k\leq \ell$. Hence $\limsup\limits_{k\to\infty} \frac{\nu^{\frac{1}{k}}_{i,k}}{k}=r_i<\infty$ for each $i=1,2,\dots, N$.  Let $r=\max\{r_1,r_2,\dots, r_N\}$. Then there is a positive constant $c>1$ such that 
\begin{align} \label{bound}
v_{k}:=\max\{v_{1,k}, v_{2,k},\dots, v_{N,k}\}<(c r)^k k^k
\end{align}
for all $k\in\mathbb{N}$.

Let $\phi(u)=\int_{\mathbb{R}^N} e^{\iota u\cdot x} dF(x)$ for any $u\in\mathbb{R}^N$. For the multi-index ${\bf m}=(m_1,m_2,\dots,m_N)$ where $m_i\in\{0,1,2,\dots\}$ for each $i=1,2,\dots, N$, we define 
\begin{align*}
|{\bf m}|=\sum^N_{i=1} m_i, \quad {\bf m}!=\prod^N_{i=1} m_i !, \quad u^{{\bf m}}=\prod^N_{i=1}u^{m_i}_i\quad \text{and}\quad  \frac{\partial^{{\bf m}} \phi}{\partial u^{{\bf m}}}=\frac{\partial^{|{\bf m}|} \phi}{\partial u^{m_1}_1u^{m_2}_2\cdots u^{m_N}_N}.
\end{align*}
Then, by Taylor's formula in several variables and the multinomial theorem, 
\begin{align*}
\left|\phi(t+u)-\sum^{n-1}_{|{\bf m}|=0}\frac{\partial^{{\bf m}} \phi(u)}{{\bf m}!}t^{{\bf m}} \right| %\leq \sum_{|{\bf m}|=n}  \frac{1}{{\bf m}!} \left(\prod^N_{i=1} |t_i|^{m_i} \right)  \int_{\mathbb{R}^N}  \left(\prod^N_{i=1} |x_i|^{m_i}\right) dF(x) \\
&\leq \int_{\mathbb{R}^N}  \sum_{|{\bf m}|=n}  \frac{1}{{\bf m}!} \left(\prod^N_{i=1} |t_i x_i|^{m_i}\right) dF(x) \\
&=\frac{1}{n!} \int_{\mathbb{R}^N} \left(\sum^N_{i=1} |t_i x_i|\right)^n dF(x) \\
&\leq  \frac{|t|^n  N^{n-1}}{n!} \sum^N_{i=1}  \int_{\mathbb{R}^N}  |x_i|^n dF(x).
\end{align*}
Using (\ref{bound}) gives
\begin{align*}
\left|\phi(t+u)-\sum^{n-1}_{|{\bf m}|=0}\frac{\partial^{{\bf m}} \phi(u)}{{\bf m}!}t^{{\bf m}} \right| 
&\leq \frac{|t|^n  N^n v_n}{n!} \leq \frac{|t|^n  (Nc r)^n n^n}{n!} \leq  (|t| N ec r)^n,
\end{align*}
where we use $e^n>n^n/n!$ in the last inequality.

Hence
\begin{align} \label{exp}
\phi(t+u)=\phi(u)+\sum^{\infty}_{|{\bf m}|=1}\frac{\partial^{{\bf m}} \phi(u)}{{\bf m}!}t^{{\bf m}}\qquad \text{for}\;  |t|\leq \frac{1}{3Ncr}.
\end{align}
Since $\phi(0)=1$, the equality (\ref{exp}) implies that the characteristic function of $F$ is uniquely determined by the given moments.
\end{proof}

%\begin{corollary} \label{uniq}
%The law of the random vector  $\big(W (L_{b_i} (\lambda) ) - W (L_{a_i} (\lambda)): 1\le i\le N\big)$ is determined by
%the moments computed in Lemma \ref{lema2}.
%\end{corollary}

\section{Acknowledge}

We would like to thank Professor Yimin Xiao for very helpful discussions on the local nondeterminism and the strong local nondeterminism for Gaussian processes.

$\begin{array}{cc}
\begin{minipage}[t]{1\textwidth}
{\bf Minhao Hong}\\
School of Science, Shanghai Maritime University, Shanghai 201306, China\\
\texttt{mhhong@shmtu.edu.cn}
\end{minipage}
\hfill
\end{array}$

\medskip

$\begin{array}{cc}
\begin{minipage}[t]{1\textwidth}
{\bf Heguang Liu}\\
School of Statistics, East China Normal University, Shanghai 200262, China \\
\texttt{1305960017@qq.com}
\end{minipage}
\hfill
\end{array}$

\medskip

$\begin{array}{cc}
\begin{minipage}[t]{1\textwidth}
{\bf Fangjun Xu}\\
KLATASDS-MOE, School of Statistics, East China Normal University, Shanghai, 200062, China \\
NYU-ECNU Institute of Mathematical Sciences at NYU Shanghai, 3663 Zhongshan Road North, Shanghai, 200062, China\\
\texttt{fjxu@finance.ecnu.edu.cn, fangjunxu@gmail.com}
\end{minipage}
\hfill
\end{array}$


\begin{thebibliography}{99}

%\bibitem{Be} Berman, S.M. (1973). Local nondeterminism and local times of Gaussian processes. \textit{Indiana Univ. Math.} \textbf{23}, 64--94.

%\bibitem{GH} Geman, D. and Horowitz, J. (1980).  Occupation densities. \textit{Annals of Probability} \textbf{8}, 1--67.


\bibitem{clrs}  Chen, X., Li, W.V.,  Rosi\'{n}ski J.  and  Shao, Q.-M. (2011). Large deviations for local times and intersection local times of fractional Brownian motions and Riemann-Liouville processes.  \textit{Ann. Probab.} \textbf{39}, 729--778.


\bibitem{cd} Cuzick, J. and Du Preez, J.P. (1982). Joint continuity of Gaussian local times. \textit{Annals of Probability}, \textbf{10}, 810--817.


\bibitem{D} Durrett, R. (2010). Probability: Theory and Examples. Cambridge University Press.

\bibitem{hx} Hong, M. and Xu, F. (2020). Derivatives of local times for some Gaussian fields. \textit{Journal of Mathematical Analysis and Applications}, \textbf{484}, 123716.


\bibitem{IL} Ibragimov, I.A. and Linnik, Y.V. (1971). Independent and stationary sequences of random variables. Wolters-Noordhoff Publishing, Groningen.
 
\bibitem{jnnp} Jaramillo, A., Nourdin, I., Nualart, D. and Peccati, G. (2023). Limit theorems for additive functionals of the fractional Brownian motion. \textit{Annals of Probability}, \textbf{51}, 1066-1111.


%\bibitem{ln} Lei, P. and Nualart, D. (2009). A decomposition of the bifractional Brownian motion and some applications. \text{Statistics and Probability Letters} \textbf{79}, 619--624.


\bibitem{l} Luan, N. (2016). Strong local non-determinism of sub-fractional Brownian motion. \textit{Appl. Math.} \textbf{6}, 2211--2216.

\bibitem{n} Nolan, J. (1989). Local nondeterminism and local times for stable processes. \textit{Probab. Th. Rel. Fields} \textbf{82}, 387-410.

\bibitem{nx} Nualart, D. and Xu, F. (2013). Central limit theorem for an additive functional of the fractional Brownian motion II. \textit{Electron. Commun. Prob.} \textbf{18}, 10pp.


%\bibitem{rc} Ruiz de Chavez, J. and Tudor, C. (2009). A decomposition of sub-fractional Brownian motion. \text{Math. Rep.}  \textbf{11}, 67--74.


\bibitem{sxy} Song, J., Xu, F. and Yu, Q. (2019). Limit theorems for functionals of two independent Gaussian processes. \textit{Stochastic Processes and their Applications}, \textbf{129}: 4791-4836.

\bibitem{tx} Tudor, C. and Xiao, Y. (2007). Sample path properties of bifractional Brownian motion. \textit{Bernoulli}, \textbf{13}, 1023--1052.


%\bibitem{xiao} Xiao, Y. (2006). Properties of local-nondeterminism of Gaussian and stable random fields and their applications. \textit{Ann. Fac. Sci. Toulouse Math.} (6), 15(1):157-193.


 

 
\end{thebibliography}
\end{document}